\documentclass{amsart}
\usepackage{amsmath, amscd}
\usepackage{stmaryrd}
\usepackage{mathrsfs}
\usepackage{amssymb}
\usepackage{amsfonts}
\usepackage{tensor}
\usepackage[all]{xy}
\usepackage{enumerate}
\usepackage{tikz}
\usepackage{adjustbox}
\usepackage{graphicx}
\usepackage{relsize}
\usepackage[bbgreekl]{mathbbol}
\usepackage{amsfonts}
\DeclareSymbolFontAlphabet{\mathbb}{AMSb} 
\DeclareSymbolFontAlphabet{\mathbbl}{bbold}
\newcommand{\Prism}{{{\mathbbl{\Delta}}}}
\usetikzlibrary{matrix, arrows, calc}
\usetikzlibrary{positioning}
\usetikzlibrary{decorations.pathreplacing}
\numberwithin{equation}{subsection}
\theoremstyle{plain}
\newtheorem{thm}[subsection]{Theorem}
\newtheorem{prop}[subsection]{Proposition}
\newtheorem{lemma}[subsection]{Lemma}
\newtheorem{cor}[subsection]{Corollary}
\theoremstyle{definition}
\newtheorem{defn}[subsection]{Definition}

\theoremstyle{remark}
\newtheorem{rem}[subsection]{Remark}

\newtheorem*{rem*}{Remark}
\newtheorem{final remark}[subsection]{Final Remark}
\newtheorem{ex}[subsection]{Example}
\makeatletter
\newcommand*\bigcdot{\mathpalette\bigcdot@{.65}}
\newcommand*\bigcdot@[2]{\mathbin{\vcenter{\hbox{\scalebox{#2}{$\m@th#1\bullet$}}}}}
\makeatother
\begin{document}
\title{Overconvergent prismatic cohomology}
\author{Andreas Langer}
\begin{abstract}
In this note I define an overconvergent version of prisms and prismatic cohomology as introduced by Bhatt and Scholze and show that overconvergent prismatic cohomology specialises to $p$-adic cohomologies, like Monsky-Washnitzer resp. rigid cohomology for smooth varieties over a perfect field, the de Rham cohomology of smooth weak formal schemes over a perfectoid ring and the \'{e}tale cohomology of its generic fibre. Besides, I give an overconvergent version of the complex $A\Omega$ of Bhatt-Morrow-Scholze and relate it to overconvergent prismatic cohomology.
\end{abstract}
\address{Harrison Building, University of Exeter, Exeter, EX4 4QF, Devon, UK}
\email{a.langer@exeter.ac.uk}
\date{August 06, 2023 \\ This research was supported by EPSRC grant EP/T005351/1}
\maketitle
\pagestyle{myheadings}

\section{Introduction}

In their fundamental work \cite{BS22} Bhatt and Scholze introduced a new $p$-adic cohomology theory, called prismatic cohomology, which has a universal character and is based on their notion of prisms and the prismatic site. Fix a prime number $p$. A prism is a pair $(A,I)$ that consists of a $\delta$-ring $A$ and an ideal $I\subset A$ satisfying some properties (\cite[Definition 1.1]{BS22}) such that $A$ is $(p,I)$-adically complete. The $\delta$-structure on $A$ induces a Frobenius lift $\varphi:A\rightarrow A$. In all cases that we consider $I$ is a principal ideal generated by a non-zero divisor $d$. The main examples in \cite{BS22} are the perfect prisms (1) and (2) and the non-perfect prism (3):
\begin{enumerate}[(1)]
\item $(W(k),(p))$ for a perfect field $k$ of characteristic $p$. Then $\varphi$ is the Witt vector Frobenius and $d=p$. (It is called the crystalline prism).
\item Let $\mathbb{C}_{p}$ be the completion of an algebraic closure of $\mathbb{Q}_{p}$ and $\mathcal{O}=\mathcal{O}_{\mathbb{C}_{p}}$ its ring of integers with tilt $\mathcal{O}^{\flat}=\displaystyle\varprojlim_{x\mapsto x^{p}}\mathcal{O}/p$. There is a natural surjection 
\begin{equation*}
\theta:A_{\mathrm{inf}}(\mathcal{O}):=W(\mathcal{O}^{\flat})\rightarrow\mathcal{O}
\end{equation*}
with $\ker\theta$ a principal ideal, and let $\varphi$ be the Witt vector Frobenius. Then $(A_{\mathrm{inf}}(\mathcal{O}),\ker\theta)$ is a perfect prism.
\end{enumerate}  
An important example of a non-perfect prism is the Breuil-Kisin prism:
\begin{enumerate}[(1)]
\setcounter{enumi}{2}
\item Let $K$ be a finite totally ramified extension of $W(k)[1/p]$. Fix a uniformiser $\pi$ and let $E(u)\in W(k)[u]$ denote the Eisenstein polynomial such that $E(\pi)=0$. Then $(W(k)\llbracket u\rrbracket,(E(u)))$ is a prism with $\varphi$ sending $u$ to $u^{p}$. 
\end{enumerate}

Now let $(A,I)$ be a prism which is bounded, i.e. $(A/I)[p^{\infty}]\cong(A/I)[p^{n}]$ for some $n$. For a smooth $p$-adic formal scheme $X$ over $\mathrm{Spf}(A/I)$ Bhatt and Scholze introduced the prismatic site $(X/A)_{\Prism}$ with structure sheaf $\mathcal{O}_{\Prism}$. An object of the prismatic site is a bounded prism $(B,I)$ over $(A,I)$ with a map $\mathrm{Spf}(B/IB)\rightarrow X$ satisfying a compatibility condition. Then the prismatic cohomology is defined as 
\begin{equation*}
R\Gamma_{\Prism}(X/A):=R\Gamma((X/A)_{\Prism},\mathcal{O}_{\Prism})\,.
\end{equation*}
The main comparison results, demonstrating the universal character of prismatic cohomology, are stated in \cite[Theorem 1.8]{BS22}. In particular, Bhatt and Scholze show that if $X$ is smooth and proper, in the above examples (1)-(3) $R\Gamma_{\Prism}(X/A)$ is a perfect $(p,d)$-complete complex in $D(A)$, equipped with a $\varphi_{A}$-linear operator $\varphi$, and 
\begin{itemize}
\item[-] a Frobenius descent of crystalline cohomology $R\Gamma_{\mathrm{cris}}(X/W(k))$ in case (1)
\item[-] a Frobenius descent of the $A_{\mathrm{inf}}$-cohomology $R\Gamma_{\mathrm{inf}}(X)$ defined in \cite{BMS18} in case (2)
\item[-] is isomorphic to Breuil-Kisin cohomology $R\Gamma_{\mathrm{BK}}(X)$ as defined in \cite{BMS19} in case (3).
\end{itemize}
Then Koshikawa and Yao \cite{Kos22}, \cite{KY23} introduced a logarithmic version of prismatic cohomology, by extending the notion of $\delta$-ring to the log context, leading to the definition of logarithmic prisms and the logarithmic prismatic site. Then they establish, in the log-smooth case, analogous comparison results for logarithmic prismatic cohomology (log-crystalline, log-de Rham, \'{e}tale, Hodge-Tate) see \cite[Theorem 2]{KY23}. They also recover the $A_{\mathrm{inf}}$-cohomology in the semistable case constructed by \v{C}esnavi\v{c}ius and Koshikawa \cite{CK19} and construct Breuil-Kisin cohomology in the semistable case by Breuil-Kisin descent along the canonical map of prisms
\begin{equation*}
(W(k)\llbracket u\rrbracket, (E(u)))\rightarrow (A_{\mathrm{inf}}(\mathcal{O}),\ker\theta); \ \ u\mapsto [\pi^{\flat}]
\end{equation*}
for a compatible system $\pi,\pi^{1/p},\ldots$ of $p$-power roots of $\pi$ in $\mathcal{O}$. 

The original motivation in our work was to impose an overconvergence condition on the prismatic site such that the hypercohomology of the corresponding overconvergent structure sheaf rationally recovers -- when the base prism is the crystalline prism -- the rigid cohomology of smooth varieties over a perfect field $k$. This naturally led to the notion of dagger prisms defined below. Although we only treat the smooth case it is probably straightforward to extend the comparison with rigid cohomology to the semistable case as well, by introducing dagger versions of the logarithmic prisms and logarithmic prismatic site of Koshikawa-Yao, which will lead to a more general definition of overconvergent logarithmic prismatic cohomology. If one replaces $p$-adically formal $\mathrm{Spf}\,\mathcal{O}$-schemes by weak formal schemes over $\mathrm{Spf}\,\mathcal{O}$, then our construction of the overconvergent prismatic site leads to dagger versions of the de Rham and \'{e}tale comparison results as stated in \cite[Theorem 1.8]{BS22}, see Theorem \ref{etale comparison intro} and Theorem \ref{q-de Rham intro} below. Using the pro-\'{e}tale site on affinoid dagger varieties we will construct a dagger version of $A_{\mathrm{inf}}$-cohomology and compare it with overconvergent prismatic cohomology. The proof relies on an almost purity property for overconvergent Witt vectors on perfectoid dagger algebras.

I consider this work as an another endeavour to introduce the notion of overconvergence into the constructions of relative $p$-adic Hodge theory. It was also inspired by the recent progress on the study of prismatic $F$-crystals and their close connection to $p$-adic Galois representations. To explain this, let $K$ be a $p$-adic field, $K_{\infty}$ the completion of the infinite cyclotomic extension $K(\mu_{p^{\infty}})$ with $K_{\infty}^{\flat}$ its tilt. Let $\Gamma_{K}=\mathrm{Gal}(K_{\infty}/K)$. Fontaine established an equivalence between the category of \'{e}tale $(\varphi,\Gamma_{K})$-modules $\mathrm{Mod}^{\varphi, \Gamma_{K},\mathrm{\acute{e}t}}_{W(K_{\infty}^{\flat})}$ and the category $\mathrm{Rep}_{\mathbb{Z}_{p}}(G_{K})$ of finite free $\mathbb{Z}_{p}$-representations of the absolute Galois group. Wu \cite{Wu21} gave a prismatic approach to this equivalence by comparing both categories with the category of prismatic $F$-crystals in $\mathcal{O}_{\Prism}[1/I]_{p}^{\wedge}$-modules over the absolute prismatic site $(\mathcal{O}_{K})_{\Prism}$. Then using results of Bhatt-Scholze \cite{BS23}, Marks \cite{Mar23} generalised the result of Wu and gave a geometric relativisation by replacing $\mathrm{Rep}_{\mathbb{Z}_{p}}(G_{K})$ by $\mathcal{O}_{K}$-local systems on the generic fibre of a formal scheme over $\mathrm{Spf}\,\mathcal{O}_{L}$ ($L$ a finite extension of $\mathbb{Q}_{p}$) and relating this to the category of Laurent $F$-crystals on the absolute prismatic site, i.e. vector bundles over a certain structure sheaf $\mathcal{O}_{\Prism}[1/I]_{\pi}^{\wedge}$, equipped with an isomorphism $\varphi^{\ast}\mathcal{M}\cong\mathcal{M}$, for a fixed uniformiser $\pi\in\mathcal{O}_{L}$. To $\pi$ one can associate a Lubin-Tate formal group $\mathcal{G}$ over $\mathcal{O}_{L}$. Then, for a $p$-adic field $K\supset L$, let $K_{\infty}$ be the $p$-adic completion of $K(\mathcal{G}[\pi^{\infty}])$, $K_{\infty}^{\flat}$ its tilt and $\Gamma_{K}=\mathrm{Gal}(K_{\infty}/K)$. Then Kisin and Ren \cite{KR09} construct a period ring $A_{K}\subset W(K_{\infty}^{\flat})\otimes_{W(\mathbb{F}_{q})}\mathcal{O}_{L}$ and extend Fontaine's theory to Lubin-Tate $(\varphi_{q},\Gamma_{K})$-modules to obtain an equivalence between the category of \'{e}tale $(\varphi_{q},\Gamma_{K})$-modules $\mathrm{Mod}_{A_{K}}^{\varphi_{q},\Gamma_{K},\mathrm{\acute{e}t}}$ and the category $\mathrm{Rep}_{\mathcal{O}_{K}}(G_{L})$ of finite free $\mathcal{O}_{K}$-representations of $G_{L}$. Marks shows that the result of Kisin-Ren is a special case of his result on Laurent $F$-crystals on the absolute prismatic site (\cite[Theorem 1.6]{Mar23}). Following ideas of Cherbonnier and Colmez \cite{CC98}, Forquaux and Xie \cite{FX13} introduced the notion of overconvergence to $K$-linear representations of $G_{L}$ and established an equivalence between the category of overconvergent $K$-representations of $G_{L}$ and \'{e}tale $(\varphi_{q},\Gamma_{K})$-modules over the Robba ring (\cite[Proposition 1.5]{FX13}). It is therefore natural to search for a category of overconvergent Laurent $F$-crystals on the (overconvergent) prismatic site of weak formal schemes and to relate it to \'{e}tale local systems on the generic fibre in a way that recovers the result of Forquaux-Xie. This will be a future project.

In the following, we give the basic definitions of dagger prisms and overconvergent prismatic cohomology and will state the main results of the paper.

We fix a perfect base prism $(A,I$) \cite[Definition 1.1 and Example 1.3]{BS22} and assume that $I=(d)$ is principal. Our main examples are $(A,I)=(W(k),(p))$, $k$ a perfect field of characteristic $p>0$, and $(A,I)=(A_{\mathrm{inf}}(\mathcal{O})=W(\mathcal{O}^{\flat}),I=(d))$ where $\mathcal{O}=\mathcal{O}_{\mathbb{C}_{p}}$, $\mathcal{O}^{\flat}$ its tilt and $d$ a generator of the kernel of the ghost map $\theta:W(\mathcal{O}^{\flat})\rightarrow\mathcal{O}$ (for example, let $\epsilon=(1,\zeta_{p},\zeta_{p^{2}},\ldots)\in\mathcal{O}^{\flat}$ and $d=1+[\epsilon^{1/p}]+\cdots+[\epsilon^{1/p}]^{p-1}$).

\begin{defn}
A dagger prism of finite type over $A$ is a pair $(S,\varphi)$ where $S$ is an $A$-algebra of finite type and $\varphi:S\rightarrow S$ is a Frobenius lift such that the following properties hold:
\begin{enumerate}[i)]
\item $S$ is $J=(p,d)$-adically separated.
\item The $J=(p,d)$-adic completion $\widehat{S}$ of $S$, together with the induced lifting $\widehat{\varphi}$ of the Frobenius, is a bounded prism and $A\rightarrow\widehat{S}$ is a map of prisms.
\item There exists finitely many elements $x_{1},\ldots, x_{r}\in S$ such that for any $s\in\widehat{S}$ we have:
\begin{itemize}
\item[-] \begin{equation*}
s=\sum_{\kappa\in\mathbb{N}^{r}}a_{\kappa}\underline{x}^{\kappa}
\end{equation*}
and for $m\in\mathbb{N}$, $a_{\kappa}\in J^{m}$ up to finitely many $\kappa$.
\item[-] $s\in S$ if and only if there exists $\epsilon>0$ and $C\in\mathbb{R}$ such that 
\begin{equation*}
\gamma_{\epsilon}(s):=\inf_{\kappa\in\mathbb{N}^{r}}(v_{J}(a_{\kappa})-\epsilon|\kappa|)>C\,.
\end{equation*}
Here $v_{J}(a_{\kappa}):=n$ if $a_{\kappa}\in J^{n}$ and $a_{\kappa}\notin J^{n+1}$.
\end{itemize}
\end{enumerate}
\end{defn}

\begin{rem}
\begin{enumerate}[i)]
\item $S$ is weakly complete with respect to the ideal $J=(p,d)$ in the sense of \cite[Definition 1.2]{MW68}
\item If $\overline{S}=S/J$ is smooth over $\overline{A}=A/J$ and $S$ is a flat $A$-algebra, then $S$ is a weak formalisation of $\overline{S}$ in the sense of \cite[Definition 3.2]{MW68}.
\item Assume $\overline{S}$ is smooth over $\overline{A}$ or a complete transversal intersection and $S$ is a weak formalisation. Then $S$ is a very smooth lifting of $\overline{S}$ to $(A,I)$ in the sense of \cite[Definition 2.5]{MW68}. In particular, given a diagram
\begin{equation*}
\begin{tikzpicture}[descr/.style={fill=white,inner sep=1.5pt}]
        \matrix (m) [
            matrix of math nodes,
            row sep=2.5em,
            column sep=2.5em,
            text height=1.5ex, text depth=0.25ex
        ]
        { \ & \ & S \\
       S & S/J & S/J \\};

        \path[overlay,->, font=\scriptsize] 
        (m-1-3) edge node[right]{$\pi$}(m-2-3)
        (m-2-1) edge (m-2-2)
        (m-2-2) edge node[above]{$\overline{\varphi}$}(m-2-3)
        ;
        
         \path[dashed,->, font=\scriptsize] 
         (m-2-1) edge (m-1-3)
         ;
                                                
\end{tikzpicture}
\end{equation*}
where $\overline{\varphi}$ is the Frobenius, there exists a lifting of Frobenius $\varphi$ on $S$ and $(\widehat{S},\varphi)$ is a prism \cite[Theorem 3.6 and Definition 2.4]{MW68}.
\item $S/dS$ is a weakly complete $A/d$-algebra with respect to the $p$-adic topology.  
\end{enumerate}
\end{rem}

\begin{defn}
An $A$-algebra is called a dagger prism if it is a direct limit $\varinjlim S_{I}$ of dagger prisms of finite type $S_{I}$ with generators $I=(x_{1},\ldots, x_{r})$.
\end{defn}

\begin{ex}
\begin{itemize}
\item[-] Let $(A,(d))=(W(k),(p))$ and $\overline{S}$ a smooth $k$-algebra. Then a Monsky-Washnitzer lift $S^{\dagger}$ of $\overline{S}$ defines a dagger prism over $(W(k),(p))$.
\item[-] Let $S$ be a dagger prism of finite type. The perfection $\varinjlim_{\varphi}S=\varinjlim(S\xrightarrow{\varphi}S\xrightarrow{\varphi}S\xrightarrow{\varphi}\cdots)$ is a dagger prism. A variant of these perfect prisms which we will call dagger perfections will be very important. For a dagger prism $S$ of finite type, we can, by Lemma \ref{Lemma} below, write $S=\varinjlim_{\epsilon}S_{\epsilon}$ where $S_{\epsilon}$ is a $(p,d)$-adically complete subring of $S$. Let $\widehat{S}$ be the associated prism. Let $\widehat{(S_{\epsilon}/p)^{\mathrm{perf}}}$ be the completion of the colimit perfection of $S_{\epsilon}/p$. Define
\begin{equation*}
W^{\dagger}(\widehat{(S_{\epsilon}/p)^{\mathrm{perf}}})=\{(z_{0},z_{1},\ldots)\in W(\widehat{(S_{\epsilon}/p)^{\mathrm{perf}}})\,:\, \varinjlim_{i}(i+\gamma_{\epsilon}(z_{i})/p^{i})=\infty \}\,.
\end{equation*}
\end{itemize}

\begin{defn}\label{0.5}
With the above notation, $S^{\dagger,\mathrm{perf}}:=\varinjlim_{\epsilon}W^{\dagger}(\widehat{(S_{\epsilon}/p)^{\mathrm{perf}}})$ is called the dagger perfection of $S$.
\end{defn}

\begin{rem}
Let $\widehat{(\widehat{S}^{\mathrm{perf}})}$ be the $(p,d)$-completed perfection of the prism $\widehat{S}$. By \cite[Lemma 3.9 and Theorem 3.10]{BS22}, $\widehat{(\widehat{S}^{\mathrm{perf}})}=W\left(\widehat{(\widehat{S}/p)^{\mathrm{perf}}}\right)$ and we have a commutative diagram
\begin{equation*}
\begin{tikzpicture}[descr/.style={fill=white,inner sep=1.5pt}]
        \matrix (m) [
            matrix of math nodes,
            row sep=2.5em,
            column sep=2.5em,
            text height=1.5ex, text depth=0.25ex
        ]
        { S & S^{\dagger,\mathrm{perf}} \\
       \widehat{S} & W\left(\widehat{(\widehat{S}/p)^{\mathrm{perf}}}\right) \\};

        \path[overlay,->, font=\scriptsize] 
        (m-1-1) edge (m-1-2)
        (m-2-1) edge (m-2-2)
        (m-1-1) edge (m-2-1)
        (m-1-2) edge ($(m-2-2)+(0,0.5)$)
        ;
                                                        
\end{tikzpicture}
\end{equation*}
Since $S^{\dagger,\mathrm{perf}}$ is weakly complete with respect to the $(p,d)$-adic topology and perfect, the map $S\rightarrow\widehat{S}\rightarrow W\left(\widehat{(\widehat{S}/p)^{\mathrm{perf}}}\right)$ factors through $S^{\dagger,\mathrm{perf}}$ and $S^{\dagger,\mathrm{perf}}$ contains the perfection $\varinjlim_{\varphi}S$ of $S$.
\end{rem}

\end{ex}

In analogy to the prismatic site \cite[\S4]{BS22}, we now define the dagger prismatic site.

\begin{defn}
Let $(A,(d))$ be a perfect prism and $X$ a $p$-adic weak formal $\mathrm{Spf}(A/d)$-scheme. The dagger prismatic site $(X/A)_{\Prism}^{\dagger}$ is the opposite category of the category of dagger prisms $(B,(d))$ over $(A,(d))$ with a map $\mathrm{Spf}^{\dagger}(B/d)\rightarrow X$, where $\mathrm{Spf}^{\dagger}(B/d)$ denotes the affine weak formal $p$-adic scheme associated to the weakly complete $A/d$-algebra $B/d$. We endow an object $(B,(d))$ with faithfully flat covers $(B,(d))\rightarrow (C,(d))$ and write an object in $(X/A)_{\Prism}^{\dagger}$ as 
\begin{equation*}
\mathrm{Spf}^{\dagger}B\leftarrow\mathrm{Spf}^{\dagger}(B/d)\rightarrow X\,.
\end{equation*}
Let $(X/A)_{\Prism}^{\dagger,\mathrm{perf}}\subset (X/A)_{\Prism}^{\dagger}$ be the full subcategory of objects $\mathrm{Spf}^{\dagger}B\leftarrow\mathrm{Spf}^{\dagger}(B/d)\rightarrow X$ for which $B$ is a dagger perfection, as in Definition \ref{0.5}.
\end{defn}

Note that the weak formal scheme $\mathrm{Spf}^{\dagger}\,B$ is well-defined: The dagger prism $B$ defines a presheaf on the underlying formal scheme $\mathrm{Spf}\,\widehat{B}$ in the obvious way and since the ideal of definition is finitely generated (generated by $p$ and $d$) one can follow Meredith's proof \cite{Mer72} to show that the presheaf is in fact a sheaf, defining $\mathrm{Spf}^{\dagger}\,B$.

Our main example is as follows:

\begin{ex}
Let $A=W(\mathcal{O}^{\flat})$, $S=A^{\dagger}\langle U_{1}^{\pm 1},\ldots, U_{d}^{\pm 1}\rangle$. Then define $A^{\dagger}\langle U_{1}^{\pm 1/p^{\infty}},\ldots, U_{d}^{\pm 1/p^{\infty}}\rangle$ as $S^{\dagger\mathrm{perf}}$. Note that there is a canonical isomorphism of the $(p,d)$-completed $A$-algebra $A\langle\underline{U}^{\pm 1/p^{\infty}}\rangle\xrightarrow{\sim}W(\mathcal{O}\langle\underline{T}^{\pm 1/p^{\infty}}\rangle^{\flat})$ given by $U_{i}^{1/p^{r}}\mapsto [(T_{i}^{1/p^{r}},T_{i}^{1/p^{r+1}},\ldots,)]$ \cite[p. 70]{BMS18}. We define $W^{\dagger}(\mathcal{O}^{\dagger}\langle T_{1}^{\pm 1/p^{\infty}},\ldots, T_{d}^{\pm 1/p^{\infty}}\rangle^{\flat})$ as the image of $A^{\dagger}\langle U_{1}^{\pm 1/p^{\infty}},\ldots, U_{d}^{\pm 1/p^{\infty}}\rangle$ under this isomorphism. It is a perfect $(p,d)$-weakly complete subalgebra of $W(\mathcal{O}\langle\underline{T}^{\pm 1/p^{\infty}}\rangle^{\flat})$.
\end{ex}

In the following we define structure sheaves on the dagger prismatic site. We first consider the affine case, so let $X=\mathrm{Spf}^{\dagger}R$ be a weak formal $\mathcal{O}$-scheme. We define two presheaves on $(R/A)^{\dagger}_{\Prism}$  by
\begin{equation*}
\mathcal{O}^{\dagger}_{\Prism}(R\rightarrow B/dB\leftarrow B)=B
\end{equation*}
and 
\begin{equation*}
\overline{\mathcal{O}}^{\dagger}_{\Prism}(R\rightarrow B/dB\leftarrow B)=B/d\,.
\end{equation*}
Both presheaves are equipped with an action of $\varphi$. We will show that $\mathcal{O}^{\dagger}_{\Prism}$ and $\overline{\mathcal{O}}^{\dagger}_{\Prism}$ are sheaves, and we point out that in the affine case presheaf cohomology computes the correct cohomology (see for example \cite[Lecture V]{Bha18a}).

\begin{defn}
\begin{enumerate}[a)]
\item We define overconvergent prismatic cohomology by 
\begin{equation*}
\Prism_{R/A}^{\dagger}:=R\Gamma((R/A)_{\Prism}^{\dagger},\mathcal{O}_{\Prism}^{\dagger})\,.
\end{equation*}
It is represented by a complex of dagger prisms over $A$ (this follows from the discussion after Lemma \ref{stacks lemma} below).
\item Dagger-Hodge-Tate cohomology is defined as
\begin{equation*}
\overline{\Prism}_{R/A}^{\dagger}:=R\Gamma((R/A)_{\Prism}^{\dagger},\overline{\mathcal{O}}_{\Prism}^{\dagger})\,.
\end{equation*}
This is represented by a complex of $p$-adically weakly complete $\mathcal{O}$-algebras.
\end{enumerate}
\end{defn}

Due to the sheaf property the definition extends to weak formal $\mathcal{O}$-schemes $X$ to get complexes $R\Gamma((X/A)_{\Prism}^{\dagger},\mathcal{O}_{\Prism}^{\dagger})$ and likewise $R\Gamma((X/A)_{\Prism}^{\dagger},\overline{\mathcal{O}}_{\Prism}^{\dagger})$.

Then we have the following main results. 

\begin{thm}(= Theorem \ref{Monsky-Washnitzer}, Corollary \ref{rigid})
\begin{enumerate}[a)]
\item Let $\overline{R}$ be a smooth $k$-algebra with Monsky-Washnitzer lift $R^{\dagger}$. Then we have a canonical isomorphism
\begin{equation*}
\Prism^{\dagger}_{\overline{R}/W(k)}\otimes\mathbb{Q}\simeq\Omega^{\bullet}_{R^{\dagger}/W(k)}\otimes\mathbb{Q}\,.
\end{equation*}
\item Let $X$ be a smooth $k$-scheme. Then we have an isomorphism
\begin{equation*}
R\Gamma((X/A)_{\Prism}^{\dagger},\mathcal{O}_{\Prism}^{\dagger})\otimes\mathbb{Q}\simeq R\Gamma_{\mathrm{rig}}(X/W(k)\otimes\mathbb{Q})\,.
\end{equation*}
\end{enumerate}
\end{thm}

\begin{thm}\label{etale comparison intro}(= Theorem \ref{etale comparison})
Let $(A,d)$ be a perfect prism, $A/d=\mathcal{O}$ and $X/\mathcal{O}$ a smooth weak formal scheme with generic fibre $X_{\eta}=X\times_{\mathrm{Spec}\,\mathcal{O}}\mathrm{Spec}\,\mathcal{O}[1/p]$. Let $\mu:(X_{\eta})_{\mathrm{\acute{e}t}}\rightarrow X_{\mathrm{\acute{e}t}}$. Then we have a canonical isomorphism
\begin{equation*}
R\mu_{\ast}\mathbb{Z}/p^{n}\simeq(\Prism^{\dagger}_{X/A}[1/d]/p^{n})^{\varphi=1}\,.
\end{equation*}
If $X=\mathrm{Spf}^{\dagger}S$ is an affine weak formal scheme with generic fibre $\mathrm{Spec}\,S[1/p]$, then we have
\begin{equation*}
R\Gamma(\mathrm{Spec}\,S[1/p],\mathbb{Z}/p^{n})\simeq(\Prism^{\dagger}_{S/A}[1/d]/p^{n})^{\varphi=1}\,.
\end{equation*}
\end{thm}

\begin{thm}\label{q-de Rham intro}(= Theorem \ref{q-de Rham})
Let $A=A_{\mathrm{inf}}(\mathcal{O})$, $X/\mathcal{O}$ a smooth weak formal scheme. Then we have an isomorphism
\begin{equation*}
\varphi^{\ast}R\Gamma((X/A)_{\Prism}^{\dagger},\mathcal{O}_{\Prism}^{\dagger})\otimes^{L}A/[p]_{q}A\cong\Omega^{\dagger\bullet}_{X/\mathcal{O}}\,.
\end{equation*}
\end{thm}

\begin{rem}
Thus theorems \ref{etale comparison intro} and \ref{q-de Rham intro} are the dagger analogues of \cite[Theorem 9.1]{BS22} and \cite[Theorem 1.8(3)]{BS22}.
\end{rem}

Finally we define an overconvergent version $A^{\dagger}\Omega_{\mathcal{X}/\mathcal{O}}$ of the complex $A\Omega_{\mathcal{X}/\mathcal{O}}$ of Bhatt-Morrow-Scholze, for a smooth weak formal $\mathcal{O}$-scheme $\mathcal{X}$ (for the definition see \S5) and prove the following comparison result.

\begin{thm}(= Theorem \ref{main theorem})
Let $\mathcal{X}$ be a smooth weak formal scheme over $\mathrm{Spf}\,\mathcal{O}$. Then we have a $\varphi$-equivariant quasi-isomorphism
\begin{equation*}
R\Gamma(\mathcal{X},A^{\dagger}\Omega_{\mathcal{X}/\mathcal{O}})\otimes^{L\dagger}_{A_{\inf}}B_{\mathrm{cris}}^{+}\cong\varphi^{\ast}\left(R\Gamma((\mathcal{X}/A_{\inf})_{\Prism},\mathcal{O}_{\Prism}^{\dagger})\right)\otimes^{L\dagger}_{A_{\inf}}B_{\mathrm{cris}}^{+}\,.
\end{equation*}
\end{thm}

In the last section we define an overconvergent version $W^{\dagger}\Omega^{\bullet}_{\mathcal{X}/\mathcal{O}}$ of the relative de Rham-Witt complex $W\Omega^{\bullet}_{\mathcal{X}/\mathcal{O}}$, generalising the construction in \cite{DLZ11}, and compare it with $A\Omega^{\dagger}_{\mathcal{X}/\mathcal{O}}$.

\begin{thm}(=Theorem \ref{overconvergent comparison})
Let $\mathcal{X}$ be a weak formal smooth $\mathcal{O}$-scheme. Then
\begin{equation*}
A\Omega^{\dagger}_{\mathcal{X}/\mathcal{O}}\otimes^{L}_{A_{\inf}}W(\mathcal{O})\cong W^{\dagger}\Omega^{\bullet}_{\mathcal{X}/\mathcal{O}}\,.
\end{equation*}
\end{thm}
\section{Perfect dagger prisms, perfectoid dagger algebras and the structure sheaf of the overconvergent prismatic site}

In this section we prove some properties of dagger prisms and in particular give a dagger version of an equivalence of categories due to Bhatt-Scholze between perfect prisms and perfectoid algebras. Finally we prove the sheaf property of the structure (pre-)sheaves $\mathcal{O}_{\Prism}^{\dagger}$ and $\overline{\mathcal{O}}_{\Prism}^{\dagger}$.

\begin{lemma}\label{Lemma}
Let $S$ be a dagger prism of finite type with generators $x_{1},\ldots, x_{r}$. Then $S=\varinjlim_{\epsilon}S_{\epsilon}$ where $S_{\epsilon}$ is a $(p,d)$-adically complete subring of $S$.
\end{lemma}
\begin{proof}
Define $A^{\dagger}\langle T_{1},\ldots, T_{r}\rangle$ to be the weak completion of $A[T_{1},\ldots, T_{r}]$ with respect to the $(p,d)$-adic topology. Then 
\begin{equation*}
\lambda\,:\,A^{\dagger}\langle T_{1},\ldots, T_{r}\rangle\rightarrow S\,;\,T_{i}\mapsto x_{i}
\end{equation*}
is surjective. Let 
\begin{equation*}
A_{\epsilon}\langle T_{1},\ldots, T_{r}\rangle=\{f\in A\langle T_{1},\ldots, T_{r}\rangle\, |\, f\text{ has radius of convergence }p^{\epsilon}, \epsilon>0,\text{ in }A[1/p]\langle T_{1},\ldots, T_{r}\rangle\}\,.
\end{equation*}
This is a $(p,d)$-adically complete subring of $A\langle T_{1},\ldots, T_{r}\rangle$ and we have $f\in A_{\epsilon}\langle T_{1},\ldots, T_{r}\rangle$ for $f=\sum_{I}a_{I}\underline{T}^{I}$, $_{I}\in A$ if and only if $\varinjlim_{I}(v_{J}(a_{I})-\epsilon|I|)=\infty$. 
$A_{\epsilon}\langle T_{1},\ldots, T_{r}\rangle=\{f\in A\langle T_{1},\ldots, T_{r}\rangle$ is equipped with the Gauss norm $\gamma_{\epsilon}(f)=\min_{I}\{v_{J}(a_{I})-\epsilon|I|\}>-\infty$ and $\,A^{\dagger}\langle T_{1},\ldots, T_{r}\rangle=\varinjlim_{\epsilon}\,A_{\epsilon}\langle T_{1},\ldots, T_{r}\rangle$. Then $S_{\epsilon}=\lambda(A_{\epsilon}\langle T_{1},\ldots, T_{r}\rangle)$ is a $(p,d)$-adically complete subring of $S$ and $S=\varinjlim_{\epsilon}S_{\epsilon}$.
\end{proof}

\begin{rem*}
Note that $S_{\epsilon}$ is not a prism because it does not admit a lifting of Frobenius (in general).
\end{rem*}

Note that the perfection of dagger prisms are already weakly complete with respect to the $(p,d)$-adic topology. Indeed, let $S=\varinjlim_{\epsilon}S_{\epsilon}$ as in Lemma \ref{Lemma}. Let $\varphi(x_{i})=x_{i}^{p}+p\delta(x_{i})$ for $i=1,\ldots, r$. Choose $\epsilon$ such that $\gamma_{\epsilon}(p\delta(x_{i}))>0$ for all $i$. This is possible since $\gamma_{\epsilon}(p)=1$. Then for any $s\in S_{\epsilon}$ we have $\varphi(s)\in S_{\epsilon/p}$ and hence
\begin{equation*}
\varinjlim_{\varphi}S=\varinjlim_{\epsilon}\varinjlim(S_{\epsilon}\xrightarrow{\varphi}S_{\epsilon/p}\xrightarrow{\varphi}S_{\epsilon/p^{2}}\xrightarrow{\varphi}\cdots)
\end{equation*}
is weakly complete.

\begin{defn}
An $\mathcal{O}$-algebra $R$ is a perfectoid dagger algebra if it is weakly complete with respect to the $\pi$-adic topology for some element $\pi\in R$ with $p\in\pi^{p}R$, the Frobenius $\varphi:R/p\rightarrow R/p$ is surjective and the $\pi$-adic completion $\widehat{R}$ is perfectoid. Moreover, we require that $\widehat{R}$ is equipped with a family of Gauss norms $\gamma_{\epsilon}$ such that $R_{\epsilon}:=\{r\in\widehat{R}\,:\,\gamma_{\epsilon}(r)\text{ is finite}\}$ is a perfectoid $\mathcal{O}$-algebra and such that $R=\varinjlim_{\epsilon}R_{\epsilon}$.
\end{defn}

Then we have a dagger version of \cite[Theorem 3.10]{BS22}:

\begin{prop}\label{equivalence of cats}
There is an equivalence of categories between dagger perfect prisms $B$ and perfectoid dagger $\mathcal{O}$-algebras $R$. The two functors are given as
\begin{equation*}
G\,:\,B \mapsto B/d 
\end{equation*}
and 
\begin{equation*}
H\,:\,R \mapsto A_{\mathrm{inf}}^{\dagger}(R)=W^{\dagger}(R^{\flat}) 
\end{equation*}
where the definition of $W^{\dagger}(R^{\flat})$ will be given below. It is a generalisation of $W^{\dagger}(\mathcal{O}^{\dagger}\langle\underline{T}^{\pm 1/p^{\infty}}\rangle^{\flat})=H(R)$ for $R=\mathcal{O}^{\dagger}\langle\underline{U}^{\pm 1/p^{\infty}}\rangle$.
\end{prop}
\begin{proof}
We assume that a perfectoid dagger $\mathcal{O}$-algebra $R$ can be written as a direct limit $R=\varinjlim_{\epsilon}R_{\epsilon}$ where $R_{\epsilon}$ is a perfectoid $\mathcal{O}$-algebra. We omit here the exact description of $R_{\epsilon}$ which will become clear when we define an equivalence of categories between perfectoid dagger $K$-algebras and perfectoid dagger $K^{\flat}$-algebras for a perfectoid field $K$ and its tilt $K^{\flat}$.

Let $R_{\epsilon}^{\flat}$ be the tilt of $R_{\epsilon}$, so $R^{\flat}_{\epsilon}=\varprojlim R_{\epsilon}/p$, and define 
\begin{equation*}
B=H(R)=\varinjlim W^{\dagger}(R_{\epsilon}^{\flat})=:A_{\mathrm{inf}}^{\dagger}(R)\,.
\end{equation*}
It is the smallest $(p,d)$-weakly complete subring of $W(\widehat{R}^{\flat})$ containing $[x]$ for $x\in R^{\flat}:=\varinjlim_{\epsilon}R^{\flat}_{\epsilon}$. 

Conversely, let $B$ be a dagger perfect prism obtained as $B=S^{\dagger\mathrm{perf}}$ for a dagger prism $S$. Then let $(S/d)^{\mathrm{perfd}}=\varinjlim_{\epsilon}(S_{\epsilon}/d)^{\mathrm{perfd}}$ be the dagger perfectoidisation of $S/d$ and $((S/d)^{\mathrm{perfd}})^{\flat}$ its tilt. Then we have 
\begin{equation*}
(S/p)^{\mathrm{perfd}}=\varinjlim_{\epsilon}\widehat{(S_{\epsilon}/p)^{\mathrm{perf}}}=((S/d)^{\mathrm{perfd}})^{\flat}
\end{equation*}
and $\varinjlim_{\epsilon}\widehat{S_{\epsilon}^{\mathrm{perf}}}$ is the smallest $(p,d)$-weakly complete perfect $A$-subalgebra of $W((\widehat{(S/d)^{\mathrm{perfd}}})^{\flat})$ containing $S$. By definition, this consists of elements $\sum p^{i}[x_{i}]^{1/p^{i}}$, $x_{i}\in(S/p)^{\mathrm{perfd}}$ such that there exists $\epsilon>0$ with $\inf_{i}\{i+\frac{1}{p^{i}}\gamma_{\epsilon}(x_{i})\}>-\infty$, where $\gamma_{\epsilon}$ is a Gauss norm with respect to the $d$-adic topology. We define
\begin{equation*}
W^{\dagger}(((S/d)^{\mathrm{perfd}})^{\flat}):= W^{\dagger}((S/p)^{\mathrm{perfd}})=\varinjlim_{\epsilon}\widehat{S_{\epsilon}^{\mathrm{perf}}}=B
\end{equation*}
and $G(B)=B/d=R$. This is $p$-adically weakly complete. Let $\pi\in\widehat{R}$ be the element satisfying $\pi^{p}\neq p$ in the perfectoid ring $\widehat{R}$. It can be chosen to being in $R$: Let $d=a_{0}\mod p$ for $a_{0}\in R^{\flat}$. Define $\pi$ to be the image of $[a_{0}^{1/p}]$ in $R$. Then $\pi^{p}$ divides $p$ in $R$. Evidently $\varphi:R/p\rightarrow R/p$ is surjective and hence $R$ is dagger perfectoid.

After taking $(p,d)$-, resp. $\pi$-, adic completion, the functors $G\circ H$ and $H\circ G$ are identities. Since $G$ maps perfect dagger prisms to perfectoid dagger $\mathcal{O}$-algebras and $H$ maps perfectoid dagger $\mathcal{O}$-algebras to perfect dagger prisms, $G\circ H$ and $H\circ G$ are again identities as functors. The proposition follows.
\end{proof}

We will compute the (pre-)sheaf cohomology by \v{C}ech-Alexander complexes, explained in the next section. Now we show that $\mathcal{O}_{\Prism}^{\dagger}$ and $\overline{\mathcal{O}}_{\Prism}^{\dagger}$ are in fact sheaves on $(R/A)_{\Prism}^{\dagger}$ with respect to the Grothendieck topology generated by faithfully flat covers.

Let $(A,d)$ be a prism, and $S$ an overconvergent prism over $(A,d)$. As in Lemma \ref{Lemma}, write $S=\varinjlim_{\epsilon}S_{\epsilon}$ where $S_{\epsilon}$ is a $(p,d)$-complete $A$-algebra. We consider the site of all $(p,d)$-complete $A$-algebras with the topology where covers are faithfully flat $(p,d)$-complete coverings. Now let $S\rightarrow B$ be a faithfully flat cover of dagger prisms. Writing $B=\displaystyle\varinjlim_{\delta\rightarrow 0}B_{\delta}$ as in Lemma \ref{Lemma}, we can assume that there exists $\delta=\phi(\epsilon)$ for a continuous function $\phi:]0,\epsilon_{0}[\rightarrow]0,\phi(\epsilon_{0})[$ with $\displaystyle\varinjlim_{\epsilon\rightarrow 0}\phi(\epsilon)=0$, such that $B_{\phi(\epsilon)}$ is faithfully flat over $S_{\epsilon}$. After renaming indices $B_{\epsilon}$ is faithfully flat over $S_{\epsilon}$. By \cite[Corollary 3.12]{BS22} the functor that sends $(S,d)\rightarrow S$ (resp. $S/d$) forms a sheaf with vanishing higher cohomology. Indeed, consider the $(p,d)$-complete \v{C}ech-nerve $B^{\bigcdot}_{\epsilon}$ of $S_{\epsilon}\rightarrow B_{\epsilon}$. By the proof of \cite[Corollary 3.12]{BS22} the corresponding total complex is exact and the sheaf property holds. Since $B^{\bigcdot}=\displaystyle\varinjlim_{\epsilon\rightarrow 0}B_{\epsilon}^{\bigcdot}$ is the weakly completed (with respect to the $(p,d)$-adic topology) \v{C}ech-nerve of $S\rightarrow B$, it  satisfies faithfully flat descent and vanishing higher cohomology as well, by exactness of the direct limit. The proof for $S/d$ is similar. 

Then we have

\begin{prop}
$\mathcal{O}_{\Prism}^{\dagger}$ and $\overline{\mathcal{O}}_{\Prism}^{\dagger}$ define sheaves on $(R/A)_{\Prism}^{\dagger}$.
\end{prop}

The proof is straightforward from the previous arguments, using weakly completed \v{C}ech-nerves of faithfully flat covers of dagger prisms. 

We have a natural map
\begin{equation*}
v:\mathrm{Shv}(R/A)_{\Prism}^{\dagger}\rightarrow\mathrm{Shv}(\mathrm{Spf}^{\dagger}R)_{\mathrm{\acute{e}t}}
\end{equation*}
which gives a complex of \'{e}tale sheaves
\begin{equation*}
\Prism^{\dagger}:=Rv_{\ast}\mathcal{O}_{\Prism}^{\dagger}\,.
\end{equation*}
Again we can define dagger prismatic cohomology as $R\Gamma(\mathrm{Spf}^{\dagger}R,\Prism^{\dagger})$.

By the obvious gluing process we obtain sheaves $\mathcal{O}_{\Prism}^{\dagger}$ and $\overline{\mathcal{O}}_{\Prism}^{\dagger}$ for any weak formal scheme $X$ and can define overconvergent prismatic cohomology as $R\Gamma(X,\Prism^{\dagger})$.

In analogy to prismatic cohomology we describe now the computation of overconvergent prismatic cohomology for smooth weak formal affine schemes $X=\mathrm{Spf}^{\dagger}R$ over $\mathrm{Spf}\,A/d$ using the \v{C}ech-Alexander complex. We recall the following lemma from \cite[Tag 07JM]{Stacks} (see also \cite[Lecture V, Lemma 4.3]{Bha18a}).

\begin{lemma}\label{stacks lemma}
Let $\mathscr{C}$ be a small category admitting finite non-empty products. Let $\mathcal{F}$ be an abelian presheaf. Assume there exists a weakly final object $X\in\mathscr{C}$, i.e. $\mathrm{Hom}(X,Y)\neq 0$ for all $Y\in\mathscr{C}$. Then $R\Gamma(\mathscr{C},\mathcal{F})$ is computed by the chain complex (the totalisation of) attached to 
\begin{equation*}
\mathcal{F}(X)\rightarrow\mathcal{F}(X\times X)\rightrightarrows\mathcal{F}(X\times X\times X)\mathrel{\substack{\textstyle\rightarrow\\[-0.4ex]
                      \textstyle\rightarrow \\[-0.4ex]
                      \textstyle\rightarrow}}\cdots
\end{equation*}
by applying $\mathcal{F}$ to the \v{C}ech-nerve of $X$.
\end{lemma}

As in \cite[Lecture V, Corollary 5.2]{Bha18a} it is proved that $(R/A)_{\Prism}^{\dagger}$ admits finite non-empty coproducts. 

There exists a free $A$-algebra (a polynomial algebra) $F_{0}$ with a surjection $F_{0}\rightarrow R$, let $J_{0}=\ker(F_{0}\rightarrow R)$ and let $F_{\delta}$ be the weakly completed free $\delta$-$A$-algebra on $F_{0}$, following the construction in \cite[4.17]{BS22}. So if $F_{0}=A[T_{1},\ldots, T_{s}]$ then $F_{\delta}=A^{\dagger}\langle T_{i},\delta(T_{i}),\delta^{2}(T_{i}),\ldots\rangle_{i=1}^{s}$. Let $J_{0}=J_{0}F_{\delta}$ be the ideal in $F_{\delta}$ generated by $J_{0}$. Then we construct the dagger prism $(F,dF)$ by taking $F_{\delta}^{\dagger}\{\frac{J_{0}}{d}\}$ where $\dagger$ means weak completion with respect to the $(p,d)$-adic topology. It has the obvious universal property. We get the following commutative diagram 
\begin{equation*}
\begin{tikzpicture}[descr/.style={fill=white,inner sep=1.5pt}]
        \matrix (m) [
            matrix of math nodes,
            row sep=2.5em,
            column sep=2.5em,
            text height=1.5ex, text depth=0.25ex
        ]
        { A & F_{\delta} & F \\
       A/d & F_{\delta}/J_{0}F_{\delta} & F/dF \\};

        \path[overlay,->, font=\scriptsize] 
        (m-1-1) edge (m-1-2)
        (m-1-2) edge (m-1-3)
        (m-1-1) edge (m-2-1)
        (m-1-2) edge (m-2-2)
        (m-1-3) edge (m-2-3)
        (m-2-1) edge (m-2-2)
        (m-2-2) edge (m-2-3)
        ;
                                                        
\end{tikzpicture}
\end{equation*}
where the maps on the top are $\delta$-maps. We obtain an object $X$ in $(R/A)_{\Prism}^{\dagger}$
\begin{equation*}
X:=(R\rightarrow F/dF\leftarrow F)\,.
\end{equation*}
Then $X$ is a weakly initial object. Indeed, for any $(R\rightarrow B/dB\leftarrow B)\in (R/A)_{\Prism}^{\dagger}$ there exists a map $F_{0}\rightarrow B$ of $A$-algebras compatible with $R\rightarrow B/dB$. It extends to a map of $\delta$-$A$-algebras $F_{\delta}\rightarrow B$ and finally to a map of dagger prisms $F\rightarrow B$ by the universal property of $F$.

By Lemma \ref{stacks lemma} $\Prism_{R/A}^{\dagger}$ is computed by the cosimplicial $\delta$-$A$-algebra 
\begin{equation*}
F^{0}\rightarrow F^{1}\rightrightarrows F^{2}\mathrel{\substack{\textstyle\rightarrow\\[-0.4ex]
                      \textstyle\rightarrow \\[-0.4ex]
                      \textstyle\rightarrow}}\cdots
\end{equation*}
where $F^{n}=\mathcal{O}_{\Prism}^{\dagger}(X^{\times (n+1)})$, so $F=F^{0}$ and $F^{n}$ is a $d$-torsion-free $(p,d)$-weakly complete $\delta$-ring.

It can be constructed as follows: Let $F_{0}^{\bigcdot}$ be the $(p,d)$-weakly completed \v{C}ech-nerve of $A\rightarrow (F_{0})^\dagger=F^{0}_{0}$ and let $J^{\bigcdot}$ be the kernel of the augmentation $F_{0}^{\bigcdot}\rightarrow F_{0}^{0}\rightarrow R$. To each $F_{0}^{\bigcdot}$ we apply the above construction of a weakly completed $\delta$-$A$-algebra $F_{\delta}^{\bigcdot}$ on $F_{0}^{\bigcdot}$ with ideal $J^{\bigcdot}F_{\delta}^{\bigcdot}$ and take the associated prism $F^{\bigcdot}=F_{\delta}^{\bigcdot\dagger}\{\frac{J^{\bigcdot}}{d}\}$ to obtain a cosimplicial object $(F^{\bigcdot}\rightarrow F^{\bigcdot}/dF^{\bigcdot}\leftarrow R)$ in $(R/A)_{\Prism}^{\dagger}$ which is the \v{C}ech-nerve of $(F^{0}\rightarrow F^{0}/dF^{0}\leftarrow R)$, the latter being the weakly final object of the topos $\mathrm{Shv}(R/A)_{\Prism}^{\dagger}$. Hence $\Prism_{R/A}^{\dagger}\cong\text{totalisation of }F^{\bigcdot}$.

In the next section we derive our first comparison theorem, namely for the base prism $(A=W(k), I=(p))$ we can compare (rational) overconvergent prismatic cohomology of $(X/A)_{\Prism}^{\dagger}$, for a smooth $k$-scheme $X$, with rigid cohomology.

\section{The comparison with Monsky-Washnitzer, respectively rigid, cohomology}

Let $\overline{A}=k[T_{1},\ldots, T_{d}]/(\overline{f}_{1},\ldots, \overline{f}_{r})$ be a smooth $k$-algebra with Monsky-Washnitzer lift $A^{\dagger}=W(k)^{\dagger}[T_{1},\ldots, T_{d}]/(f_{1},\ldots, f_{r})$. Consider the following frame over $W(k)$:
\begin{equation*}
B=W(k)[T_{1},\ldots, T_{d}][Y_{1},\ldots, Y_{r}]/(f_{1}-pY_{1},\ldots, f_{r}-pY_{r})\,.
\end{equation*}
Then $B$ is an Elkik lift of $\overline{B}=B/p=\overline{A}[Y_{1},\ldots, Y_{r}]$. Let $B^{\dagger}$ be the Monsky-Washnitzer (= weak) completion of $B$, which is a Monsky-Washnitzer lift of $\overline{B}$. We have
\begin{equation*}
\Omega^{\bullet}_{B^{\dagger}/W(k)}\otimes\mathbb{Q}\cong\Omega^{\bullet}_{A^{\dagger}/W(k)}\otimes\mathbb{Q}
\end{equation*}
by \cite[Theorem 5.4]{MW68}. Now we use the ideas in \cite[Construction 4.17, 4.18]{BS22} to compute overconvergent prismatic cohomology using a \v{C}ech-Alexander complex.

Let $M^{\bigcdot}$ be the weak completed \v{C}ech-nerve of $W(k)\rightarrow M^{0}=W(k)[T_{1},\ldots, T_{d}]$, so $M^{n}=(W(k)[T_{1},\ldots, T_{d}])^{\otimes(n+1)})^{\dagger}$. Let $J^{\bigcdot}$ be the kernel of $M^{\bigcdot}\rightarrow M^{0}\rightarrow\overline{A}$. Take for each $M^{\bigcdot}$ the associated weakly completed free $\delta$-$W(k)$-algebra $F_{\delta}^{\bigcdot}$ on $M^{\bigcdot}$ with ideal $J^{\bigcdot}F_{\delta}^{\bigcdot}$ and let $B^{\dagger\bigcdot}=F^{\bigcdot\dagger}_{\delta}\langle\frac{J^{\bigcdot}}{p}\rangle$, which is the \v{C}ech-nerve of $B^{\dagger 0}=F_{\delta}^{0\dagger}\langle\frac{J^{0}}{p}\rangle$. Then $(B^{\dagger\bigcdot}\rightarrow B^{\dagger\bigcdot}/p\leftarrow\overline{A})$ is the \v{C}ech-nerve as cosimplicial object of $(B^{\dagger 0}\rightarrow B^{\dagger 0}/p\leftarrow\overline{A})$ in $(\overline{A}/W(k))_{\Prism}^{\dagger}$. Then it follows from Lemma \ref{stacks lemma} and the subsequent construction that any object in $(\overline{A}/W(k))_{\Prism}^{\dagger}$ receives a map from $(B^{\dagger 0}\rightarrow B^{\dagger 0}/p\leftarrow\overline{A})$ hence the latter object is a weakly initial object in $(\overline{A}/W(k))_{\Prism}^{\dagger}$. This implies that $\Prism^{\dagger}_{\overline{A}/W(k)}$ is computed by $B^{\dagger\bigcdot}$. Let $f_{1},\ldots, f_{r(n)}$ be generators of the ideal $J^{n}$. Then 
\begin{equation*}
B^{\dagger n}=F^{n\dagger}_{\delta}\langle Y_{1},\ldots, Y_{r(n)}\rangle/(f_{1}-pY_{1},\ldots, f_{r(n)}-pY_{r(n)})\,.
\end{equation*}
Then $B^{\dagger n}\otimes_{W(k)}k=\overline{A}[Z_{1},Z_{2},\ldots]$ is a polynomial algebra over $\overline{A}$ in infinitely many variables. It is then clear that $B^{\dagger n}$ is a direct limit of Monsky-Washnitzer lifts of polynomial algebras over $\overline{A}$ in finitely many variables. Applying the argument \cite[Theorem 5.4]{MW68} as at the beginning of this section we conclude that
\begin{equation*}
\Omega^{\bullet}_{B^{\dagger n}/W(k)}\otimes\mathbb{Q}\cong\Omega^{\bullet}_{A^{\dagger}/W(k)}\otimes\mathbb{Q}
\end{equation*}
for all $n$. We will see that an argument analogous to \cite{BdJ11} implies that 
\begin{equation*}
\Prism_{\overline{A}/W(k)}^{\dagger}\otimes\mathbb{Q}\cong \Omega^{\bullet}_{A^{\dagger}/W(k)}\otimes\mathbb{Q}\,.
\end{equation*}
Indeed, let $M^{r,s}=\Omega^{r}_{B^{\dagger s}/W(k)}\otimes\mathbb{Q}$, and consider $M^{\bullet, \bullet}$ as a first quadrant double complex. 

\begin{lemma}\label{homotopic to zero}
For $i>0$ the cosimplicial module $\Omega^{i}_{B^{\dagger\bigcdot}/W(k)}\otimes\mathbb{Q}$ is homotopy equivalent to zero.
\end{lemma}
\begin{proof}
This holds for the cosimplicial module $\Omega^{i}_{M^{\bigcdot}/W(k)}$ by taking the weak completion of $\Omega^{i}_{P^{\bigcdot}/W(k)}$ where $P^{\bigcdot}$ is the \v{C}ech-nerve of $W(k)\rightarrow W(k)[T_{1},\ldots, T_{d}]$ and using \cite[Lemma 2.15, Lemma 2.17]{BdJ11}. The same argument implies the statement for $\Omega^{i}_{F_{\delta}^{\bigcdot}/W(k)}$ and then one has to tensor the cosimplicial module $\Omega^{i}_{F_{\delta}^{\bigcdot}/W(k)}$ with the $F_{\delta}^{\bigcdot}$-module $B^{\dagger\bigcdot}$ to prove the lemma.
\end{proof}

Each column complex $M^{\bullet,s}$ is quasi-isomorphic to $M^{\bullet,0}=\Omega^{\bullet}_{A^{\dagger}/W(k)}\otimes\mathbb{Q}$. The complex $M^{0,\bullet}$ computes $\Prism_{\overline{A}/W(k)}^{\dagger}\otimes\mathbb{Q}$. The total complex $\mathrm{Tot}(M^{\bullet,\bullet})$ computes the cohomology of the de Rham complex $\Omega^{\bullet}_{A^{\dagger}/W(k)}\otimes\mathbb{Q}$ by the first spectral sequence associated to the double complex $M^{\bullet, \bullet}$. On the other hand, by Lemma \ref{homotopic to zero}, $\mathrm{Tot}(M^{\bullet,\bullet})$ also computes the cohomology of $M^{0,\bullet}$ by the second spectral sequence. Hence we have proven:

\begin{thm}\label{Monsky-Washnitzer}
Let $\overline{R}$ be a smooth $k$-algebra with Monsky-Washnitzer lift $R^{\dagger}$. Then we have a canonical isomorphism 
\begin{equation*}
\Prism_{\overline{R}/W(k)}^{\dagger}\otimes\mathbb{Q}\cong\Omega^{\bullet}_{R^{\dagger}/W(k)}\otimes\mathbb{Q}\,.
\end{equation*}
\end{thm}

Now we globalise the above comparison: Let $\{U_{i}\}$ be an affine covering, $U_{i}=\mathrm{Spec}\,A_{i}\subset X$, of $X$, a smooth $k$-scheme. $R\Gamma(X,\Prism_{X/W(k)}^{\dagger})$ is the total complex of 
\begin{equation*}
\prod_{i}\Prism_{U_{i}/W(k)}^{\dagger}\rightrightarrows\prod_{i,j}\Prism_{U_{i}\cap U_{j}/W(k)}^{\dagger}\mathrel{\substack{\textstyle\rightarrow\\[-0.4ex]
                      \textstyle\rightarrow \\[-0.4ex]
                      \textstyle\rightarrow}}\cdots\,.
\end{equation*}
On the other hand, we have a commutative diagram
\begin{equation*}
\begin{tikzpicture}
    \node (A) at (0,0) {$\displaystyle\prod_{i}\Prism_{U_{i}/W(k)}^{\dagger}\otimes\mathbb{Q}$};
    \node (B) at (4,0) {$\displaystyle\prod_{i,j}\Prism_{U_{i}\cap U_{j}/W(k)}^{\dagger}\otimes\mathbb{Q}$};
    \node (C) at (0,-2) {$\displaystyle\prod_{i}W^{\dagger}\Omega^{\bullet}_{U_{i}/k}\otimes\mathbb{Q}$};
    \node (D) at (4,-2) {$\displaystyle\prod_{i,j}W^{\dagger}\Omega^{\bullet}_{U_{i}\cap U_{j}/k}\otimes\mathbb{Q}$};
    \node (E) at (7,0) {$\cdots$};
    \node (F) at (7,-2) {$\cdots$};
    \draw[transform canvas={yshift=0.9ex},->] (A) -- (B);
    \draw[transform canvas={yshift=0.2ex},->] (A) -- (B);
    \draw[transform canvas={yshift=0.9ex},->] (C) -- (D);
    \draw[transform canvas={yshift=0.2ex},->] (C) -- (D);
    \draw[transform canvas={yshift=0.9ex},->] (B) -- (E);
    \draw[transform canvas={yshift=0.2ex},->] (B) -- (E);
    \draw[transform canvas={yshift=1.6ex},->] (B) -- (E);
    \draw[transform canvas={yshift=0.9ex},->] (D) -- (F);
    \draw[transform canvas={yshift=0.2ex},->] (D) -- (F);
    \draw[transform canvas={yshift=1.6ex},->] (D) -- (F);
    \path[overlay,->, font=\scriptsize] 
        (A) edge node[left]{$\simeq$}(C)
        (B) edge node[right]{$\simeq$}(D);
\end{tikzpicture} 
\end{equation*}
where the isomorphisms come from \cite[Corollary 3.25]{DLZ11}: we have isomorphisms 
\begin{equation*}
\Prism_{\mathrm{Spec}\,\overline{B}/W(k)}^{\dagger}\otimes\mathbb{Q}\simeq\Omega^{\bullet}_{B^{\dagger}/W(k)}\otimes\mathbb{Q}\simeq W^{\dagger}\Omega^{\bullet}_{\overline{B}/k}\otimes\mathbb{Q}
\end{equation*}
for any smooth $k$-algebra $\overline{B}$ with Monsky-Washnitzer lift $B^{\dagger}$. This implies the existence of a map
\begin{equation*}
\begin{tikzpicture}
 \node (G) at (-1.6,0) {$R\Gamma(X,\Prism^{\dagger}_{X/W(k)}\otimes\mathbb{Q})=$};
    \node (A) at (2,0) {$\mathrm{Tot}\biggl(\displaystyle\prod_{i}\Prism_{U_{i}/W(k)}^{\dagger}\otimes\mathbb{Q}$};
    \node (B) at (6,0) {$\displaystyle\prod_{i,j}\Prism_{U_{i}\cap U_{j}/W(k)}^{\dagger}\otimes\mathbb{Q}$};
     \node (E) at (9,0) {$\cdots\biggr)$};
     \node (C) at (2,-2) {$\mathrm{Tot}\biggl(\displaystyle\prod_{i}W^{\dagger}\Omega^{\bullet}_{U_{i}/k}\otimes\mathbb{Q}$};
    \node (D) at (6,-2) {$\displaystyle\prod_{i,j}W^{\dagger}\Omega^{\bullet}_{U_{i}\cap U_{j}/k}\otimes\mathbb{Q}$};
    \node (F) at (9,-2) {$\cdots\biggr)$};
    \node (H) at (-1.6,-2) {$R\Gamma(X,W^{\dagger}\Omega^{\bullet}_{X/k}\otimes\mathbb{Q})=$};
    \draw[transform canvas={yshift=0.9ex},->] (A) -- (B);
    \draw[transform canvas={yshift=0.2ex},->] (A) -- (B);
    \draw[transform canvas={yshift=0.9ex},->] (C) -- (D);
    \draw[transform canvas={yshift=0.2ex},->] (C) -- (D);
    \draw[transform canvas={yshift=1.2ex},->] (B) -- (E);
    \draw[transform canvas={yshift=0.7ex},->] (B) -- (E);
    \draw[transform canvas={yshift=0.2ex},->] (B) -- (E);
    \draw[transform canvas={yshift=1.2ex},->] (D) -- (F);
    \draw[transform canvas={yshift=0.7ex},->] (D) -- (F);
    \draw[transform canvas={yshift=0.2ex},->] (D) -- (F);
    \path[overlay,->, font=\scriptsize] 
        (A) edge node[left]{$\simeq$}(C)
        (B) edge node[right]{$\simeq$}(D)
        (G) edge (H);
       
       \end{tikzpicture} 
\end{equation*}
due to the sheaf properties of $\Prism^{\dagger}_{X/W(k)}$ and $W^{\dagger}\Omega^{\bullet}_{X/k}$, which by construction is an isomorphism.

\begin{cor}\label{rigid}
Let $X$ be a smooth $k$-scheme. Then we have an isomorphism
\begin{equation*}
R\Gamma(X,\Prism^{\dagger}_{X/W(k)}\otimes\mathbb{Q})\cong R\Gamma_{\mathrm{rig}}(X/W(k)\otimes\mathbb{Q})
\end{equation*}
between rational overconvergent prismatic and rigid cohomology.
\end{cor}
\begin{proof}
See \cite[Theorem 4.40]{DLZ11} for the case that $X$ is quasiprojective, and \cite{Law18} for the general case.
\end{proof}

\section{\'{E}tale comparison of overconvergent prismatic cohomology}

Let $(A,d)$ be a perfect prism, $A/d=\mathcal{O}$, and $X/\mathcal{O}$ a smooth weak formal scheme with generic fibre $X_{\eta}=X\times_{\mathrm{Spec}\,\mathcal{O}}\mathrm{Spec}\,\mathcal{O}[1/p]$. Then we have
\begin{thm}\label{etale comparison}
Let $\mu:(X_{\eta})_{\mathrm{\acute{e}t}}\rightarrow (X)_{\mathrm{\acute{e}t}}$. Then we have a canonical isomorphism
\begin{equation*}
R\mu_{\ast}\mathbb{Z}/p^{n}\cong\left(\Prism_{X/A}^{\dagger}[1/d]/p^{n}\right)^{\varphi=1}\,.
\end{equation*}
If $X=\mathrm{Spf}^{\dagger}S$ is affine with generic fibre $\mathrm{Spec}\,S[1/p]$ an affinoid dagger variety, then we have
\begin{equation*}
R\Gamma(\mathrm{Spec}\,S[1/p],\mathbb{Z}/p^{n})\cong\left(\Prism_{S/A}^{\dagger}[1/d]/p^{n}\right)^{\varphi=1}\,.
\end{equation*}
\end{thm}

\begin{rem}\label{remark 3.2}
\begin{itemize}
\item[-] Theorem \ref{etale comparison} is the dagger analogue of \cite[Theorem 9.1]{BS22}.

\item Let $\widehat{X}=\mathrm{Spf}\,\widehat{S}$ be the formal smooth scheme associated to $\mathrm{Spf}^{\dagger}S$ with generic fibre $\mathrm{Spec}\,\widehat{S}[1/p]$, a rigid variety. Then 
\begin{equation*}
R\Gamma(\mathrm{Spec}\,S[1/p],\mathbb{Z}/p^{n})\cong R\Gamma(\mathrm{Spec}\,\widehat{S}[1/p],\mathbb{Z}/p^{n})\,.
\end{equation*}
The \'{e}tale cohomology of an affinoid dagger variety coincides with the \'{e}tale cohomology of its associated affinoid variety \cite[Propositon 3.5]{CN20}. Hence the theorem implies that $\left(\Prism_{S/A}^{\dagger}[1/d]/p^{n}\right)^{\varphi=1}$ is isomorphic to $\left(\Prism_{\widehat{S}/A}[1/d]/p^{n}\right)^{\varphi=1}$.
\end{itemize}
\end{rem}

\begin{proof}
Let $C^{\bigcdot}$ be the \v{C}ech-Alexander complex that computes $\Prism_{S/A}^{\dagger}$. We have a simplicial object $(C^{\bigcdot}\rightarrow C^{\bigcdot}/dC^{\bigcdot}\leftarrow S)$ in $(S/A)^{\dagger}_{\Prism}$. Each $C^{i}$ is a dagger prism and is weakly complete with respect to the $(p,d)$-adic topology: $C^{i}=(B_{\delta}^{i})^{\dagger}\{\frac{J^{i}}{d}\}$ where $B_{\delta}^{i}$ is the $(p,d)$-weakly completed $\delta$-$B^{i}$-algebra associated to $B^{i}$ and where $B^{\bigcdot}$ is the \v{C}ech-nerve of $A\rightarrow B^{0}$, a $(p,d)$-weakly complete free $A$-algebra with kernel $J^{0}=\ker(B^{0}\rightarrow S)$ and $J^{\bigcdot}=\ker(B^{\bigcdot}\rightarrow S)$. Then $C^{i}=\displaystyle\varinjlim_{\epsilon\rightarrow 0}C^{i}_{\epsilon}$ (as in Lemma \ref{Lemma}). The family of Gauss norms $\gamma_{\epsilon}$ defining $C^{i}_{\epsilon}$ can be chosen in a compatible way on $C^{\bigcdot}$ to get simplicial complexes $C^{\bigcdot}_{\epsilon}$ of $(p,d)$-complete $A$-algebras with $\varinjlim_{\epsilon}C_{\epsilon}^{\bigcdot}=C^{\bigcdot}$. Then $C_{\epsilon}^{\bigcdot}/p$ is $d$-adically complete. Using \cite[Lemma 9.2]{BS22}, we see that $C_{\epsilon}^{\bigcdot}/p\in D_{\mathrm{comp}}(\mathcal{O}^{\flat}[F])$ and taking $\varinjlim_{\epsilon}C_{\epsilon}^{\bigcdot}/p$, $\epsilon>0$, commutes with the functors $M\mapsto M^{\varphi=1}$ and $M\mapsto M[1/d]^{\varphi=1}$, hence
\begin{equation*}
(\Prism^{\dagger}_{S/A}[1/d]/p)^{\varphi=1}=\varinjlim_{\epsilon\rightarrow 0}(C_{\epsilon}^{\bigcdot}/p)[1/d]^{\varphi=1}\,.
\end{equation*}
Define $\Prism_{S/A,\mathrm{perf}}^{\dagger}:=C^{\bullet\dagger,\mathrm{perf}}$ using Definition \ref{0.5}. We have
\begin{equation*}
(\Prism_{S/A}^{\dagger}[1/d]/p)^{\varphi=1}\xrightarrow[\varphi]{\sim}(\Prism_{S/A}^{\dagger}/p[1/d])^{\varphi=1}
\end{equation*}
hence we get
\begin{equation*}
(\Prism_{S/A}^{\dagger}[1/d]/p^{n})^{\varphi=1}\xrightarrow{\sim}(\Prism_{S/A,\mathrm{perf}}^{\dagger}[1/d]/p^{n})^{\varphi=1}\,.
\end{equation*}
Indeed, without loss of generality $n=1$. Then
\begin{equation*}
(\Prism_{S/A}^{\dagger}[1/d]/p)^{\varphi=1}=\varinjlim_{\epsilon\rightarrow 0}(C_{\epsilon}/p[1/d])^{\varphi=1}
\end{equation*}
and 
\begin{equation*}
\Prism^{\dagger}_{S/A,\mathrm{perf}}[1/d]/p)^{\varphi=1}=\displaystyle\varinjlim_{\epsilon\rightarrow 0}((C_{\epsilon}^{\bullet}/p)^{\mathrm{perfd}}[1/d])^{\varphi=1}
\end{equation*}
where $(C_{\epsilon}^{\bullet}/p)^{\mathrm{perfd}}$ is the $d$-completed filtered colimit of $C_{\epsilon}^{\bullet}/p\xrightarrow{\varphi}C_{\epsilon}^{\bullet}/p\xrightarrow{\varphi}\cdots$. As each map acts trivially on applying $((-)[1/d])^{\varphi=1}$, \cite[Lemma 9.2]{BS22} implies that 
\begin{align*}
(\Prism_{S/A,\mathrm{perf}}^{\dagger}[1/d]/p)^{\varphi=1}
& = \varinjlim_{\epsilon\rightarrow 0}(C^{\bullet}_{\epsilon}/p[1/d])^{\varphi=1} \\
& = (\Prism^{\dagger}_{S/A}[1/d]/p)^{\varphi=1}
\end{align*}
as claimed.

Let $S$ be a dagger perfectoid with $p$-adic completion $\widehat{S}$. Consider the composite map
\begin{equation*}
\alpha\,:\,\Prism_{S/A}^{\dagger}\rightarrow\Prism_{\widehat{S}/A}\cong W(\widehat{S}^{\flat})
\end{equation*}
where the isomorphism is derived from \cite[Lemma 4.7]{BS22}. Since $(W^{\dagger}(S^{\flat})\rightarrow W^{\dagger}(S^{\flat})/d=S)$ is an object in $(\mathrm{Spf}^{\dagger}S/A)_{\Prism}$, there exists a unique map $\Prism_{S/A}^{\dagger}\rightarrow W^{\dagger}(S^{\flat})$ compatible with $\Prism_{S/A}\xrightarrow{\sim} W(S^{\flat})$. Then the analogue of \cite[Lemma 4.7]{BS22} also holds: Let $S\rightarrow B/d$ be a map, for a dagger prism $B$ with completion $\widehat{B}$ we have a commutative diagram
\begin{equation*}
\begin{tikzpicture}[descr/.style={fill=white,inner sep=1.5pt}]
        \matrix (m) [
            matrix of math nodes,
            row sep=2.5em,
            column sep=2.5em,
            text height=1.5ex, text depth=0.25ex
        ]
        { W^{\dagger}(S^{\flat}) & B \\
       W(S^{\flat}) & \widehat{B} \\};

        \path[overlay,->, font=\scriptsize] 
        (m-2-1) edge (m-2-2)
        (m-1-1) edge (m-2-1)
        (m-1-2) edge (m-2-2)
        ;
        
        \path[dashed,->, font=\scriptsize] 
        (m-1-1) edge (m-1-2)
        ;                    
                                                        
\end{tikzpicture}
\end{equation*}
where the lower horizontal arrow is induced by \cite[Lemma 4.7]{BS22} and the upper horizontal arrow is induced by the canonical map $\Prism^{\dagger}_{S/A}\rightarrow B$ which exists by the universal property of $\Prism^{\dagger}_{S/A}$. Hence $\Prism^{\dagger}_{S/A}\cong W^{\dagger}(S^{\flat})$.

In the following we reduce the proof of Theorem \ref{etale comparison} to the case that $S$ is dagger perfectoid. 

For an affine weak formal scheme $\mathrm{Spf}\,S$ with formal completion $\mathrm{Spf}\,\widehat{S}$ we have an isomorphism (Remark \ref{remark 3.2})
\begin{equation*}
F(S):=R\Gamma(\mathrm{Spec}\,S[1/p],\mathbb{Z}/p^{n})\cong R\Gamma(\mathrm{Spec}\,\widehat{S}[1/p],\mathbb{Z}/p^{n})=:F(\widehat{S})
\end{equation*}
and we write $S=\varinjlim_{\epsilon}S_{\epsilon}$ where $S_{\epsilon}$ is $p$-adically complete and $\{S_{\epsilon}\}_{\epsilon}$ defines the dagger structure on $\mathrm{Spf}\,\widehat{S}=\mathrm{Spf}\,S$. Then
\begin{equation*}
F(S)=\varinjlim_{\epsilon}R\Gamma(\mathrm{Spec}\,S_{\epsilon}[1/p],\mathbb{Z}/p^{n})=\varinjlim_{\epsilon}F(S_{\epsilon})\,.
\end{equation*}
In the proof of \cite[Theorem 9.1]{BS22} a comparison map
\begin{equation*}
F(-)\rightarrow G(-):=(\Prism_{-/A}[1/d]/p^{n})^{\varphi=1}
\end{equation*}
of $\mathrm{arc}_{p}$-sheaves is constructed on $\mathrm{fSch}_{/\mathrm{Spf}\,\mathcal{O}}$ and shown to be an isomorphism.

\begin{defn}
A homomorphism of weakly complete $\mathcal{O}$-algebras $S\rightarrow T$ is called an $\mathrm{arc}_{p}$-cover if $T=\varinjlim_{\epsilon}T_{\epsilon}$ has a dagger presentation with $p$-complete $\mathcal{O}$-algebras $T_{\epsilon}$ such that $S_{\epsilon}\rightarrow T_{\epsilon}$ is an $\mathrm{arc}_{p}$-cover in the sense of \cite[Definition 6.14]{BM21}.
\end{defn}
Note that $F$ is an $\mathrm{arc}_{p}$-sheaf by \cite[Corollary 6.17]{BM21} on $\mathrm{Spf}\,S_{\epsilon}$, hence $F(S)=\varinjlim_{\epsilon}F(S_{\epsilon})$ is an $\mathrm{arc}_{p}$-sheaf on $\mathrm{Spf}\,S$.

Now let again $C^{\bullet}$ be the \v{C}ech-Alexander complex computing $\Prism^{\dagger}_{S/A}$ and $\Prism_{S/A,\mathrm{perf}}^{\dagger}=C^{\bullet \dagger,\mathrm{perf}}$ as above. Then we have
\begin{align*}
\Prism^{\dagger}_{S/A,\mathrm{perf}}/p^{n}
& =C^{\bullet \dagger,\mathrm{perf}}/p^{n} \\
& =\varinjlim_{\epsilon\rightarrow 0}W^{\dagger}((C^{\bullet}_{\epsilon}/p)^{\mathrm{perfd}})/p^{n} \\
& =\varinjlim_{\epsilon\rightarrow 0}W_{n}((C^{\bullet}_{\epsilon}/p)^{\mathrm{perfd}}) \\
& =\varinjlim_{\epsilon\rightarrow 0}\Prism_{S_{\epsilon}/A,\mathrm{perf}}/p^{n}\,.
\end{align*}
This implies that
\begin{equation*}
(\Prism^{\dagger}_{S/A,\mathrm{perf}}[1/d]/p^{n})^{\varphi=1}=\varinjlim_{\epsilon\rightarrow 0}(\Prism_{S_{\epsilon}/A,\mathrm{perf}}[1/d]/p^{n})^{\varphi=1}\,.
\end{equation*}
Since $(\Prism_{S_{\epsilon}/A,\mathrm{perf}}[1/d]/p^{n})^{\varphi=1}$ is an $\mathrm{arc}_{p}$-sheaf on $\mathrm{Spf}\,S_{\epsilon}$, $(\Prism^{\dagger}_{S/A,\mathrm{perf}}[1/d]/p^{n})^{\varphi=1}$ is an $\mathrm{arc}_{p}$-sheaf on $\mathrm{Spf}\,S$.

One defines the $\mathrm{arc}$-topology for weak formal $p$-adic schemes by using $p$-adic weakly complete valuation rings of rank $1$ instead, and shows that there exists an $\mathrm{arc}$-cover $S\rightarrow T$ with $T$ dagger perfectoid. The proof is entirely similar to \cite[\S8]{BS22} for $p$-adic formal schemes. The proof of \cite[Proposition 8.10]{BS22} transfers to affine perfectoid dagger schemes $\mathrm{Spf}\,T$; that is, $H^{i}_{\mathrm{arc}}(\mathrm{Spf}\,T,\mathcal{O})=0$ and the structure presheaf on $\mathrm{Spf}\,T$ is a sheaf. 

In order to show that the map
\begin{equation*}
F(S)=\varinjlim_{\epsilon}F(S_{\epsilon})\rightarrow(\Prism^{\dagger}_{S/A,\mathrm{perf}}[1/d]/p^{n})^{\varphi=1}\simeq\varinjlim_{\epsilon\rightarrow 0}(\Prism_{S_{\epsilon}/A,\mathrm{perf}}[1/d]/p^{n})^{\varphi=1}
\end{equation*}
is an isomorphism, we can work $\mathrm{arc}$-locally and hence from now on we may assume that $S$ is dagger perfectoid. We claim that we have an isomorphism
\begin{equation}\label{tilt isomorphism}
R\Gamma(\mathrm{Spec}\,S[1/p],\mathbb{Z}/p^{n})\cong R\Gamma(\mathrm{Spec}\,S^{\flat}[1/d],\mathbb{Z}/p^{n})\,.
\end{equation}
Indeed, we have 
\begin{align*}
R\Gamma(\mathrm{Spec}\,S^{\flat}[1/d],\mathbb{Z}/p^{n})
& =R\Gamma(\mathrm{Spa}\,S^{\flat}[1/d],\mathbb{Z}/p^{n}) \\
& =R\Gamma(\mathrm{Spa}\,\widehat{S}^{\flat}[1/d],\mathbb{Z}/p^{n}) \\
& =R\Gamma(\mathrm{Spec}\,\widehat{S}^{\flat}[1/d],\mathbb{Z}/p^{n})
\end{align*}
and likewise for $S$. This follows from the comparison theorem of Huber/Scholze for \'{e}tale cohomology of rigid/adic spaces. The isomorphism
\begin{equation*}
R\Gamma(\mathrm{Spa}\,\widehat{S}[1/p],\mathbb{Z}/p^{n})\cong R\Gamma(\mathrm{Spa}\,\widehat{S}^{\flat}[1/d],\mathbb{Z}/p^{n})
\end{equation*}
follows from \cite[Theorem 1.11]{Sch12}. This shows that we have \eqref{tilt isomorphism}.

It remains to check 
\begin{align*}
R\Gamma(\mathrm{Spec}\,S^{\flat}[1/d],\mathbb{Z}/p^{n})
& = (\Prism^{\dagger}_{S/A}[1/d]/p^{n})^{\varphi=1} \\
& = (W_{n}(S^{\flat})[1/d])^{\varphi=1}\,.
\end{align*}
By Artin-Schreier-Witt we have
\begin{equation*}
\mathbb{Z}/p^{n}\cong (W_{n}(S^{\flat})[1/d])^{\varphi=1}
\end{equation*}
and since by the dagger analogue of \cite[Theorem 4.9, Lemma 4.10]{Sch13} we have
\begin{equation*}
W_{n}(S^{\flat})[1/d]\xrightarrow{\varphi-1}W_{n}(S^{\flat})[1/d]
\end{equation*}
is surjective, we conclude that $H^{i}(\mathrm{Spa}\,S^{\flat}[1/d],\mathbb{Z}/p^{n})$ vanishes for $i>0$. This concludes the proof of Theorem \ref{etale comparison}.
\end{proof}

\section{The tilting correspondence for perfectoid dagger algebras}\label{tilting correspondence section}
Let $K$ be a perfectoid field, with absolute value $|\,\,\,|$, ring of integers $\mathcal{O}$ in $K$, and $w$ an element with $|w|=1/p$. Consider the ring $R=K\langle T_{1}^{1/p^{\infty}},\ldots, T_{d}^{1/p^{\infty}}\rangle$ of $w$-adically converging power series, so $z=\sum_{I\in\mathbb{N}[1/p]^{d}}a_{I}\underline{T}^{I}$, $a_{I}\in K$ is in $K\langle T_{1}^{1/p^{\infty}},\ldots, T_{d}^{1/p^{\infty}}\rangle$ if and only if $|a_{I}|\rightarrow 0$. Then $K\langle T_{1}^{1/p^{\infty}},\ldots, T_{d}^{1/p^{\infty}}\rangle$ is the perfectoidisation of the Tate algebra $K\langle T_{1},\ldots, T_{d}\rangle$, equipped with the Gauss norm $\|z\|=\max_{I}\{|a_{I}|\}$. Let $R^{\circ}\subset R$ be the subring with $a_{I}\in\mathcal{O}$ and let $\lambda\in\mathcal{O}$, $|\lambda|<1$, we write $\lambda=w^{\epsilon}$, $\epsilon>0$. Then 
\begin{equation*}
\left\{z=\sum_{I\in\mathbb{N}[1/p]^{d}}a_{I}\underline{T}^{I}\,:\, \varinjlim_{I}|a_{I}||\lambda|^{-|I|}=0\right\}
\end{equation*}
is for every $\lambda=w^{\epsilon}$ a Banach algebra $R_{\epsilon}$ with Gauss norm $\max\{|a_{I}|w^{\epsilon}|^{-|I|}\}$. We write 
\begin{equation*}
R_{\epsilon}=K\langle (\lambda T_{1})^{1/p^{\infty}},\ldots, (\lambda T_{d})^{1/p^{\infty}}\rangle\,,
\end{equation*}
so we have convergence for $|\lambda T_{i}|\leq 1$. We have an isomorphism 
\begin{align*}
& R_{\epsilon}\rightarrow K\langle S_{1}^{1/p^{\infty}},\ldots, S_{d}^{1/p^{\infty}}\rangle \\
& \, \,  T_{i}\mapsto \lambda^{-1}S_{i}, \text{ (so }S_{i}=\lambda T_{i})\,.
\end{align*}
Let $|\lambda|<|\lambda'|<1$, $\lambda=w^{\epsilon}$, $\lambda'=w^{\epsilon'}$. Then we have an inclusion
\begin{equation*}
\begin{tikzpicture}[descr/.style={fill=white,inner sep=1.5pt}]
        \matrix (m) [
            matrix of math nodes,
            row sep=2.5em,
            column sep=2.5em,
            text height=1.5ex, text depth=0.25ex
        ]
        { R_{\epsilon} & R_{\epsilon'} \\
       K\langle S_{1}^{1/p^{\infty}},\ldots, S_{d}^{1/p^{\infty}}\rangle & K\langle {S'}_{1}^{1/p^{\infty}},\ldots, {S'}_{d}^{1/p^{\infty}}\rangle \\
       S_{i} & \frac{\lambda}{\lambda'}S'_{i} \\};

        \path[overlay,->, font=\scriptsize] 
        (m-1-1) edge (m-1-2)
        (m-2-1) edge (m-2-2)
        (m-1-1) edge node[right]{$\wr$} (m-2-1)
        (m-1-2) edge node[right]{$\wr$} (m-2-2)
        ;
        
        \path[overlay,|->, font=\scriptsize]
        (m-3-1) edge (m-3-2)
        ;
                                                                
\end{tikzpicture}
\end{equation*}
and we obtain a map of power bounded elements
\begin{equation*}
\mathcal{O}\langle S_{1}^{1/p^{\infty}},\ldots, S_{d}^{1/p^{\infty}}\rangle\rightarrow\mathcal{O}\langle {S'}_{1}^{1/p^{\infty}},\ldots, {S'}_{d}^{1/p^{\infty}}\rangle
\end{equation*}
inducing $R^{\circ}_{\epsilon}\rightarrow R_{\epsilon'}^{\circ}$. Define $S=\varinjlim_{\epsilon}R_{\epsilon}=\bigcup_{\epsilon}R_{\epsilon}$. We define
\begin{equation*}
S^{\circ}=\{f\in S\, :\, \forall\alpha\in\mathcal{O}, |\alpha|<1, \exists\lambda=w^{\epsilon}\text{ such that }f\in R_{\epsilon}\text{ and }\alpha f\in R^{\circ}_{\epsilon}\}\,.
\end{equation*}
Now let $K$ be a perfectoid field of characteristic $0$, $K^{\flat}$ its tilt (of characteristic $p>0$). Let $\mathcal{O}$, resp. $\mathcal{O}^{\flat}$, be the ring of integers in $K$, resp. in $K^{\flat}$. We construct perfectoid dagger algebras over $K$, resp. over $K^{\flat}$, and derive a tilting correspondence. Let $A/K^{\flat}$ be a reduced affinoid dagger algebra with integral elements $A^{+}/\mathcal{O}^{\flat}$. We fix a presentation $(K^{\flat})^{\dagger}\langle T_{1},\ldots, T_{d}\rangle\rightarrow A$ (and $(\mathcal{O}^{\flat})^{+}\langle T_{1},\ldots, T_{d}\rangle\rightarrow A^{+}$). Define Gauss norms $\gamma_{\epsilon}$ on $\widehat{A}$ via the presentation and let $A_{\epsilon}=\{a\in A\,:\, \gamma_{\epsilon}(a)>-\infty\}$. $A_{\epsilon}$ is a $\varpi$-adically complete Tate algebra over $K^{\flat}$, with integral elements $A^{+}_{\epsilon}$ over $\mathcal{O}^{\flat}$. Let $\widehat{A}_{\epsilon}^{\mathrm{perf}}$ be the $\varpi$-adically completed perfection, with integral elements $\widehat{A}_{\epsilon}^{+\mathrm{perf}}$. This is a perfectoid $K^{\flat}$-algebra with 
\begin{equation*}
\mathrm{Spa}\,(\widehat{A}_{\epsilon}^{\mathrm{perf}},\widehat{A}_{\epsilon}^{+\mathrm{perf}})=\mathrm{Spa}\,(\widehat{A}_{\epsilon},\widehat{A}_{\epsilon}^{+})
\end{equation*}
(see \cite[Proposition 6.11]{Sch12}). Define
\begin{equation*}
(A^{\mathrm{perf}},A^{+\mathrm{perf}})=\varinjlim_{\epsilon}(\widehat{A}^{\mathrm{perf}}_{\epsilon},\widehat{A}_{\epsilon}^{+\mathrm{perf}})\,.
\end{equation*} 
This is a perfectoid dagger algebra over $K^{\flat}$ with 
\begin{equation*}
\mathrm{Spa}\,(A^{\mathrm{perf}},A^{+\mathrm{perf}})=\mathrm{Spa}\,(\widehat{A},\widehat{A}^{+})\,.
\end{equation*}
The chosen presentation extends to a presentation $R_{\epsilon}\twoheadrightarrow\widehat{A}_{\mathrm{\epsilon}}^{\mathrm{perf}}$, $R^{\circ}_{\epsilon}\twoheadrightarrow\widehat{A}_{\mathrm{\epsilon}}^{\mathrm{+perf}}$, and hence $S=\varinjlim_{\epsilon}R_{\epsilon}\twoheadrightarrow A^{\mathrm{perf}}$, $S^{\circ}\twoheadrightarrow A^{+\mathrm{perf}}$. Let $(B_{\epsilon},B_{\epsilon}^{+})$ be the untilt of $(\widehat{A}^{\mathrm{perf}}_{\epsilon},\widehat{A}_{\epsilon}^{+\mathrm{perf}})$ under Scholze's tilting correspondence. Then $\mathrm{Spa}\,(B_{\epsilon},B_{\epsilon}^{+})$ is homeomorphic to its tilt $\mathrm{Spa}\,(B_{\epsilon},B_{\epsilon}^{+})^{\flat}=\mathrm{Spa}\,(\widehat{A}_{\epsilon}^{\mathrm{perf}},\widehat{A}_{\epsilon}^{+\mathrm{perf}})$ \cite[Theorem 6.3]{Sch12}.

Define $(B,B^{+})=\varinjlim_{\epsilon}(B_{\epsilon},B_{\epsilon}^{+})$. This is a perfectoid dagger algebra over $K$. Then we have the analogue of \cite[Proposition 6.17]{Sch12}:

\begin{thm}
There is an equivalence of perfectoid dagger spaces over $K$ and perfectoid dagger spaces over $K^{\flat}$, by associating to a perfectoid dagger space $X^{\dagger}/K$ its tilt $(X^{\flat})^{\dagger}/K^{\flat}$.
\end{thm}
\begin{proof}
The reduction to the case of an affinoid dagger space is shown in the same way as in \cite[Proposition 6.17]{Sch12}. For any affinoid perfectoid dagger algebra $(B,B^{+})$ with tilt $(A^{\mathrm{perf}},A^{+\mathrm{perf}})$, the associated presheaves $\mathcal{O}_{X}$ over $X=\mathrm{Spa}(B,B^{+})$ and $\mathcal{O}_{X^{\flat}}$ on $X^{\flat}=\mathrm{Spa}(B,B^{+})^{\flat}$, defined by taking direct limits of the restrictions of $\mathcal{O}_{X_{\epsilon}}$ on $X_{\epsilon}=\mathrm{Spa}(B_{\epsilon},B_{\epsilon}^{+})$ and $\mathcal{O}_{X_{\epsilon}^{\flat}}$ on $X^{\flat}_{\epsilon}=\mathrm{Spa}(B_{\epsilon},B_{\epsilon}^{+})^{\flat}$ to $X$ resp. $X^{\flat}$ are sheaves by \cite[Proposition 6.14]{Sch12}, and \cite[Theorem 6.3]{Sch12} holds verbatim for affinoid perfectoid dagger spaces. In particular, we have a homeomorphism
\begin{equation*}
\mathrm{Spa}(B,B^{+})\cong\mathrm{Spa}(B,B^{+})^{\flat} \ (=\mathrm{Spa}(\widehat{B},\widehat{B}^{+})=\mathrm{Spa}(\widehat{B},\widehat{B}^{+})^{\flat})
\end{equation*}
which commutes with the homeomorphism between the associated affinoid perfectoid spaces. For an affinoid perfectoid space the equivalence follows from the above construction, namely
\begin{equation*}
\mathrm{Spa}\,(B,B^{+})^{\flat}=\mathrm{Spa}\,(A^{\mathrm{perf}},A^{+\mathrm{perf}})\,.
\end{equation*}
\end{proof}

The main example is 
\begin{equation*}
\mathrm{Spa}\,(K^{\dagger}\langle T_{1}^{1/p^{\infty}},\ldots,T_{d}^{1/p^{\infty}}\rangle,\mathcal{O}^{\dagger}\langle T_{1}^{1/p^{\infty}},\ldots,T_{d}^{1/p^{\infty}}\rangle )
\end{equation*}
and its tilt
\begin{equation*}
\mathrm{Spa}\,(K^{\flat\dagger}\langle T_{1}^{1/p^{\infty}},\ldots,T_{d}^{1/p^{\infty}}\rangle,\mathcal{O}^{\flat\dagger}\langle T_{1}^{1/p^{\infty}},\ldots,T_{d}^{1/p^{\infty}}\rangle )
\end{equation*}
where $\dagger$ means $\varpi$-adic weak completion.

\section{A dagger version of $A\Omega$}

Let $(A,I)=(A_{\mathrm{inf}}(\mathcal{O}),I=(d))$. Let $\mathcal{X}/\mathrm{Spf}\,\mathcal{O}$ be a weak formal smooth scheme and let $X$ be its generic fibre, a dagger variety. Let $X_{\mathrm{pro\acute{e}t}}$ be the pro-\'{e}tale site on $X$, given locally by 
\begin{equation*}
(U_{i+1}\xrightarrow{\mathrm{f\acute{e}t}}U_{i}\xrightarrow{\mathrm{f\acute{e}t}}\cdots\xrightarrow{\mathrm{f\acute{e}t}}U_{1}\xrightarrow{\mathrm{\acute{e}t}}X)
\end{equation*}
where $U_{i}$ are affinoid dagger varieties, $U_{i+1}\rightarrow U_{i}$ is finite \'{e}tale for $i\geq 1$ and $U_{1}\rightarrow X$ is \'{e}tale. Let $U=\varprojlim_{i}\mathrm{Spec}\,U_{i}$. Let $T_{i}=\varinjlim_{\epsilon}T_{\epsilon,i}$ be the affinoid dagger algebra corresponding to $U_{i}$, with the notation of \S\ref{tilting correspondence section}. Let $T_{\epsilon,i}^{\mathrm{perf}}$ be the perfectoidisation arising from the $T_{\epsilon,i}$, with integral elements $T^{\circ,\mathrm{perf}}_{\epsilon,i}$. We have the structure sheaf $\mathcal{O}(U)=\varinjlim_{\epsilon}T_{\epsilon}^{\mathrm{perf}}$, and $\mathcal{O}^{+}(U)=\varinjlim_{\epsilon}T^{\circ,\mathrm{perf}}_{\epsilon}$. These algebras are weakly complete with respect to the $p$-adic topology. Then $(\mathcal{O}(U),\mathcal{O}^{+}(U))$ is a perfectoid dagger algebra over $\mathbb{C}_{p}$. Let $Z=\mathrm{Spa}\,(\mathcal{O}(U),\mathcal{O}^{+}(U))$ and let $Z^{\flat}=\mathrm{Spa}\,(\mathcal{O}(U^{\flat}),\mathcal{O}^{\flat+}(U))$ be its tilt. Let $v:X_{\mathrm{pro\acute{e}t}}\rightarrow\mathcal{X}_{\mathrm{Zar}}$ be the canonical map of topoi. 

We want to define a dagger version of $A\Omega_{\widehat{X}/\mathcal{O}}$ where $\widehat{X}$ is the associated formal smooth scheme and $A\Omega_{\widehat{X}/\mathcal{O}}$ is defined as in \cite{BMS18}. So we first need to give a reasonable definition of $A_{\inf,X}^{\dagger}=W^{\dagger}(\mathcal{O}_{X}^{+\flat})$. Let us first consider the case $\mathcal{X}=\mathrm{Spf}^{\dagger}R$ for $R=\mathcal{O}^{\dagger}\langle T_{1}^{\pm 1},\ldots, T_{d}^{\pm 1}\rangle$. We have an isomorphism
\begin{equation*}
A_{\inf}\langle\underline{U}^{\pm 1/p^{\infty}}\rangle\cong W(\mathcal{O}\langle\underline{T}^{\pm 1/p^{\infty}}\rangle^{\flat})
\end{equation*}
of complete $A_{\inf}$-algebras \cite[p. 70]{BMS18} with respect to the $(p,d)$-adic topology. Define the $(p,d)$-weakly complete analogue
\begin{align}
A^{\dagger}_{\inf}\langle\underline{U}^{\pm 1/p^{\infty}}\rangle
& =\varinjlim_{\epsilon\rightarrow 0}A_{\inf,\epsilon}\langle\underline{U}^{\pm 1/p^{\infty}}\rangle \\
& =\varinjlim_{\epsilon\rightarrow 0}A_{\inf}\langle(p^{\epsilon}\underline{U})^{\pm 1/p^{\infty}}\rangle \nonumber \\
& =\varinjlim_{\epsilon\rightarrow 0}W^{\dagger}(\mathcal{O}\langle(p^{\epsilon}\underline{T})^{\pm 1/p^{\infty}}\rangle^{\flat}) \nonumber \\
& =:W^{\dagger}(\mathcal{O}^{\dagger}\langle\underline{T}^{\pm 1/p^{\infty}}\rangle^{\flat}) \nonumber\,.
\end{align}
To be precise, $A_{\inf}\langle(p^{\epsilon}\underline{U})^{\pm 1/p^{\infty}}\rangle$ are the power series in $\underline{U}^{\pm 1/p^{\infty}}$ with radius of convergence $p^{\epsilon}$ in $A_{\inf}[1/p]\langle\underline{U}^{\pm 1/p^{\infty}}\rangle$. Then  $W^{\dagger}(\mathcal{O}^{\dagger}\langle\underline{T}^{\pm 1/p^{\infty}}\rangle^{\flat})$ consists of Witt vectors $(w_{0},w_{1},\ldots)$ such that there exists $\epsilon>0$ with $\inf_{i}\{i+\gamma_{\epsilon}(w_{i})/p^{i}\}>-\infty$. In particular, $w_{i}\in\mathcal{O}\langle(p^{\epsilon}\underline{T})^{\pm 1/p^{\infty}}\rangle^{\flat}$.

For $R$ \'{e}tale over $\mathcal{O}^{\dagger}\langle T_{1}^{\pm 1},\ldots, T_{d}^{\pm 1}\rangle$ we have, using a lifting $A^{\dagger}(R)$ of $R$ over $A_{\inf}$ under the map $\theta:A_{\inf}\rightarrow\mathcal{O}$, 
\begin{equation}
A^{\dagger}(R)\otimes_{A_{\inf}^{\dagger}\langle\underline{U}^{\pm 1}\rangle}A_{\inf}^{\dagger}\langle\underline{U}^{\pm 1/p^{\infty}}\rangle=A_{\inf}^{\dagger}(R_{\infty})=:W^{\dagger}((R_{\infty})^{\flat})
\end{equation}
where $(R_{\infty})^{\flat}=(R\otimes_{\mathcal{O}^{\dagger}\langle\underline{T}^{\pm 1}\rangle}\mathcal{O}^{\dagger}\langle\underline{T}^{\pm 1/p^{\infty}}\rangle)^{\flat}$. (This is a dagger version of \cite[p. 70]{BMS18}).

This construction sheafifies (see also the proof of Proposition \ref{almost zero} below), namely we have  
\begin{equation*}
W^{\dagger}(\mathcal{O}_{X}^{+\flat})=\displaystyle\varinjlim_{\epsilon\rightarrow 0}W^{\dagger}(\mathcal{O}_{X_{\epsilon}}^{+\flat})
\end{equation*}
for $X_{\epsilon}=\mathrm{Spf}\,R_{\epsilon}$ (such that $R_{\epsilon}$ is \'{e}tale over $\mathcal{O}\langle p^{\epsilon}\underline{T}^{\pm 1}\rangle$ and $R=\varinjlim_{\epsilon}R_{\epsilon}$) and global sections equal to $W^{\dagger}((R_{\infty})^{\flat})$. We define for any $X=\mathrm{Spec}\,S$ the sheaf $A^{\dagger}_{\mathrm{inf},X}$ on $X_{\mathrm{pro\acute{e}t}}$
\begin{equation*}
A^{\dagger}_{\mathrm{inf},X}:=W^{\dagger}(\mathcal{O}_{X}^{+\flat})
\end{equation*}
by local conditions using a covering by small affines and using that $A_{\inf, X}$ is a sheaf on $\widehat{X}_{\mathrm{pro\acute{e}t}}$.

\begin{defn}
We define a dagger version of $A\Omega$ as follows: Let $\mathcal{X}$ be a weak formal smooth scheme over $\mathcal{O}$. Then 
\begin{equation*}
A^{\dagger}\Omega_{\mathcal{X}/\mathcal{O}}=L\eta_{\mu}(Rv_{\ast}A^{\dagger}_{\inf, X})
\end{equation*}
where $L\eta_{\mu}$ is the d\'{e}calage functor and $\mu=[\epsilon]-1\in W(\mathcal{O}^{\flat})$.
\end{defn}

Now let $\mathcal{X}=\mathrm{Spf}^{\dagger}R$ be small affine as above. We have a local-to-global map
\begin{equation*}
A^{\dagger}\Omega^{\mathrm{pro\acute{e}t}}_{\mathcal{X}/\mathcal{O}}:=L\eta_{\mu}R\Gamma_{\mathrm{pro\acute{e}t}}(\mathcal{X},A^{\dagger}_{\inf, X})\rightarrow R\Gamma(\mathcal{X},A^{\dagger}\Omega_{\mathcal{X}/\mathcal{O}})\,.
\end{equation*}

\begin{lemma}
The local-to-global map is a quasi-isomorphism.
\end{lemma}
\begin{proof}
Since $L\eta_{\mu}$ commutes with direct limits, we have:
\begin{align*}
A^{\dagger}\Omega^{\mathrm{pro\acute{e}t}}_{\mathcal{X}/\mathcal{O}}
& \cong\varinjlim_{\epsilon}L\eta_{\mu}R\Gamma_{\mathrm{pro\acute{e}t}}(\mathcal{X}_{\epsilon},A^{\dagger}_{\inf, X_{\epsilon}}) \\
& \cong\varinjlim_{\epsilon}R\Gamma(\mathcal{X}_{\epsilon},A^{\dagger}\Omega_{\mathcal{X}_{\epsilon}/\mathcal{O}}) \\
& \cong R\Gamma(\mathcal{X},A^{\dagger}\Omega_{\mathcal{X}/\mathcal{O}})
\end{align*}
where the middle quasi-isomorphism follows from \cite[Proposition 9.14]{BMS18}.
\end{proof}

Our main theorem for comparing overconvergent prismatic cohomology with $A^{\dagger}\Omega$ can be formulated as follows:

\begin{thm}\label{main theorem}
Let $\mathcal{X}$ be weak formal smooth over $\mathrm{Spf}\,\mathcal{O}$. Then we have a $\varphi$-equivariant quasi-isomorphism
\begin{equation*}
R\Gamma(\mathcal{X},A^{\dagger}\Omega_{\mathcal{X}/\mathcal{O}})\otimes^{L\dagger}_{A_{\inf}}B_{\mathrm{cris}}^{+}\cong\varphi^{\ast}\left(R\Gamma((\mathcal{X}/A_{\inf})_{\Prism},\mathcal{O}_{\Prism}^{\dagger})\right)\otimes^{L\dagger}_{A_{\inf}}B_{\mathrm{cris}}^{+}
\end{equation*}
where $B_{\mathrm{cris}}^{+}$ is Fontaine's ring.
\end{thm}
\begin{proof}
It suffices to show the analogous statement for $A^{\dagger}\Omega^{\mathrm{pro\acute{e}t}}_{\mathcal{X}/\mathcal{O}}$ where $\mathcal{X}=\mathrm{Spf}^{\dagger}R$ is small affine. We view $A_{\inf}(\mathcal{O})=W(\mathcal{O}^{\flat})$ as a perfect $\mathbb{Z}\llbracket q-1\rrbracket$-algebra, for $q=[\epsilon]$, $\epsilon=(1,\zeta_{p},\zeta_{p^{2}},\ldots)\in\mathcal{O}^{\flat}$. The outline of the proof is as follows: We relate $A^{\dagger}\Omega^{\mathrm{pro\acute{e}t}}_{\mathcal{X}/\mathcal{O}}$ to group cohomology, where the group is $\Gamma=\mathbb{Z}_{p}(1)^{d}$, the Galois group of the pro-\'{e}tale extension $\mathcal{U}/U$. The group cohomology can be computed by a dagger Koszul complex, which can be related to a dagger $q$-de Rham complex that computes overconvergent prismatic cohomology.

To be precise, let $R$ be small \'{e}tale over $\mathcal{O}^{\dagger}\langle T_{1}^{\pm 1},\ldots, T_{d}^{\pm 1}\rangle$ and $R_{i}=R\otimes_{\mathcal{O}^{\dagger}\langle\underline{T}^{\pm 1}\rangle}\mathcal{O}^{\dagger}\langle\underline{T}^{\pm 1/p^{i}}\rangle$. Then $\mathcal{U}=``\varprojlim"\mathrm{Spec}\,R_{i}[1/p]$ is a Galois cover of $U=\mathrm{Spec}\,R[1/p]$. We have an action of $\mu_{p^{i}}^{d}$ on $R_{i}[1/p]$ in the usual way: $\underline{\xi}=(\xi_{1},\ldots,\xi_{d})\in\mu_{p^{i}}^{d}$ acts via $\underline{\xi}T_{1}^{i_{1}/p^{i}}\cdots T_{d}^{i_{d}/p^{i}}=\xi_{1}^{i_{1}}\cdots\xi_{d}^{i_{d}}T_{1}^{i_{1}/p^{i}}\cdots T_{d}^{i_{d}/p^{i}}$. We get a Cartan-Leray spectral sequence
\begin{equation*}
H^{n}_{\mathrm{cont}}(\mu_{p^{i}}^{d},H^{l}_{\mathrm{pro\acute{e}t}}(U_{i},W^{\dagger}(\mathcal{O}^{+\flat}_{X}))\Rightarrow H^{n+l}_{\mathrm{pro\acute{e}t}}(U,W^{\dagger}(\mathcal{O}^{+\flat}_{X}))
\end{equation*}
for $U_{i}=\mathrm{Spa}\,R[1/p]$ and after taking limits over $i$ yielding the derived version
\begin{equation*}
R\Gamma_{\mathrm{gp}}(\mathbb{Z}_{p}(1)^{d},R\Gamma_{\mathrm{pro\acute{e}t}}(\mathcal{U},W^{\dagger}(\mathcal{O}^{+\flat}_{X}))\simeq R\Gamma_{\mathrm{pro\acute{e}t}}(U,W^{\dagger}(\mathcal{O}^{+\flat}_{X}))\,.
\end{equation*}

We have the following proposition, which is an almost purity version for overconvergent Witt vectors:

\begin{prop}\label{almost zero}
\begin{itemize}
\item[-] $H^{i}_{\mathrm{pro\acute{e}t}}(\mathcal{U},W^{\dagger}(\mathcal{O}^{+\flat}_{X}))$ is almost zero, killed by $W(\mathfrak{m}^{\flat})$, for $i>0$.
\item[-] $H^{0}_{\mathrm{pro\acute{e}t}}(\mathcal{U},W^{\dagger}(\mathcal{O}^{+\flat}_{X}))=W^{\dagger}((R_{\infty})^{\flat})$ where $(R_{\infty})^{\flat}$ is the dagger tilt of the perfectoid dagger algebra $R_{\infty}$.
\end{itemize}
\end{prop}
\begin{proof}
As before, let $W^{\dagger}(\mathcal{O}_{X}^{+\flat})=\displaystyle\varinjlim_{\epsilon\rightarrow 0}W^{\dagger}(\mathcal{O}_{X_{\epsilon}}^{+\flat})$ for $\mathcal{X}_{\epsilon}=\mathrm{Spf}\,R_{\epsilon}$. By \cite[4.10, 5.11]{Sch13} we have $H^{i}_{\mathrm{pro\acute{e}t}}(\mathcal{U},W(\mathcal{O}^{+\flat}_{X_{\epsilon}}))$ is almost zero for $i>0$ and $H^{0}_{\mathrm{pro\acute{e}t}}(\mathcal{U},W(\mathcal{O}^{+\flat}_{X_{\epsilon}}))=W(R_{\infty,\epsilon}^{\flat})$. We need the same assertions for the overconvergent Witt vectors. Now let $S^{\circ}=R_{\infty,\epsilon}^{\flat}$ be as above, which is a perfectoid algebra over $\mathcal{O}_{K^{\flat}}$ and $S=R_{\infty,\epsilon}[1/p]^{\flat}$ which is a perfectoid $K^{\flat}$-algebra. Let $S'/S$ be finite \'{e}tale of degree $n$, then $S'$ is a perfectoid $K^{\flat}$-algebra and $S'^{\circ}$, the integral closure of $S^{\circ}$ in $S'$, is almost finite \'{e}tale over $S^{\circ}$ by \cite[Proposition 5.23]{Sch12}. By \cite[Corollary 2.46]{DLZ12} $W^{\dagger}(S')$ is finite \'{e}tale over $W^{\dagger}(S)$ of degree $n$. The proof in \cite{DLZ12}, which is given for a finite \'{e}tale extension of a finitely generated algebra over a perfect field, easily transfers to perfectoid algebras. Under different assumptions the result is also proved in \cite{DK15}. 

Let $X_{\epsilon}^{\flat}=\mathrm{Spa}(\mathcal{O}_{X_{\epsilon}}^{\flat},\mathcal{O}_{X_{\epsilon}}^{+\flat})$. We will show that $W^{\dagger}(\mathcal{O}_{X_{\epsilon}}^{\flat})$ is a sheaf on $X^{\flat}_{\epsilon}$ for the topology generated by rational subdomains. Moreover, for any \'{e}tale covering $\{U_{\epsilon,i}\}$ of $X_{\epsilon}^{\flat}$, where $U_{\epsilon,i}$ is a rational subdomain of a finite \'{e}tale map to a rational subdomain of $X_{\epsilon}^{\flat}$, using \'{e}tale acyclicity of the sheaf $\mathcal{O}_{X_{\epsilon}}^{\flat}$ (\cite[Proposition 7.13]{Sch12}), the total complex associated to the simplicial complex
\begin{equation}\label{simplicial}
0\rightarrow\varinjlim_{\epsilon}W^{\dagger}(\mathcal{O}(X_{\epsilon}^{\flat}))\rightarrow\prod_{i}\varinjlim_{\epsilon}W^{\dagger}(\mathcal{O}(U_{\epsilon,i}))\rightrightarrows\prod_{i,j}\varinjlim_{\epsilon}W^{\dagger}(\mathcal{O}(U_{\epsilon,i}\times U_{\epsilon,j}))\mathrel{\substack{\textstyle\rightarrow\\[-0.4ex]
                      \textstyle\rightarrow \\[-0.4ex]
                      \textstyle\rightarrow}}\cdots
\end{equation}
is exact. Indeed, let $S_{\epsilon}=(R_{\infty,\epsilon}[1/p])^{\flat}$, a perfectoid $K^{\flat}$-algebra, and let $S'_{\epsilon}$ be finite \'{e}tale, again a perfectoid $K^{\flat}$-algebra. We may assume that there are surjections (with the notation of \S\ref{tilting correspondence section})
\begin{equation*}
\rho_{\epsilon}:K^{\flat}\langle w^{\epsilon}T_{1}^{1/p^{\infty}},\ldots, w^{\epsilon}T_{d}^{1/p^{\infty}}\rangle\twoheadrightarrow S_{\epsilon}
\end{equation*}
and likewise for $S'_{\epsilon}$. Define finite Gauss norms $\gamma_{S_{\epsilon}}$ on $S_{\epsilon}$, resp. $\gamma_{S'_{\epsilon}}$ on $S'_{\epsilon}$, as follows: let $w^{\epsilon}t_{i}$ be the image of $w^{\epsilon}T_{i}$ in $S_{\epsilon}$ and define $\gamma^{\rho_{\epsilon}}(z)=\inf_{I}(v(a_{I})-\epsilon|k_{I}|)>-\infty$ for $z=\sum_{k_{I}\in\mathbb{N}[1/p]^{d}}a_{I}(w^{\epsilon}t_{I})^{k_{I}}\in S_{\epsilon}$ and $|k_{I}|=\sum_{i}k_{i}$, and then $\gamma_{S_{\epsilon}}(z):=\sup_{\rho_{\epsilon}}\gamma^{\rho_{\epsilon}}(z)$, and likewise for $S'_{\epsilon}$. Now let $U_{\epsilon}$ be a rational subdomain of $\mathrm{Spa}\,S'_{\epsilon}$, hence
\begin{align*}
U_{\epsilon}
& =\mathrm{Spa}\,S'_{\epsilon}\langle (w^{\epsilon}Y_{1})^{1/p^{\infty}},\ldots, (w^{\epsilon}Y_{s})^{1/p^{\infty}}\rangle/\langle f_{i}-gY_{i}\rangle \\
& =\mathrm{Spa}\,S'_{\epsilon}\left\langle \frac{(w^{\epsilon}f_{1})^{1/p^{\infty}},\ldots, (w^{\epsilon}f_{s})^{1/p^{\infty}}}{g^{1/p^{\infty}}}\right\rangle \\
& =\mathrm{Spa}\,S'_{U_{\epsilon}}\,.
\end{align*}
Define the Gauss norms $\gamma_{S'_{U_{\epsilon}}}$ by adding a degree valuation for each new variable $Y_{i}$. It then follows from the techniques in \cite[\S1]{DLZ12} that the restriction of $\gamma_{S'_{U_{\epsilon}}}$ to $S_{\epsilon}$ is linearly equivalent to $\gamma_{S_{\epsilon}}$, that is, there exists $c>0$ such that 
\begin{equation*}
\gamma_{S'_{U_{\epsilon}}}(z)\geq\gamma_{S_{\epsilon}}(z)\geq c\gamma_{S'_{U_{\epsilon}}}(z)
\end{equation*}
for all $z\in S_{\epsilon}$. We may assume that $\gamma_{S'_{U_{\epsilon}}}(z)$ is negative, so that $c>1$. Then $\frac{1}{c}\gamma_{S_{\epsilon}}(z)\geq\gamma_{S'_{U_{\epsilon}}}(z)$, but from the definitions we have the inequality $\gamma_{S_{\epsilon/c}}(z)\geq \frac{1}{c}\gamma_{S_{\epsilon}}(z)$ and hence $\gamma_{S_{\epsilon/c}}(z)\geq\gamma_{S'_{U_{\epsilon}}}(z)$ for all $z\in S_{\epsilon}$. Now let $(z_0,z_1,\ldots)\in W(S_{\epsilon})$ be a Witt vector such that its image lies in $W^{\dagger}(S'_{U_{\epsilon}})$. Then $\varinjlim_{i}\left(i+\frac{\gamma_{S'_{U_{\epsilon}}}(z_{i})}{p^{i}}\right)=\infty$, hence $(z_0,z_1,\ldots)\in W^{\dagger}(S_{\epsilon/c})$. To show acyclicity in higher degrees it suffices to consider the two cases
\begin{enumerate}[i)]
\item $U_{\epsilon}=\mathrm{Spa}\,S'_{\epsilon}$ finite \'{e}tale faithfully flat over $\mathrm{Spa}\,S_{\epsilon}$. Then the complex \eqref{simplicial} associated to the acyclic complex
\begin{equation*}
0\rightarrow S_{\epsilon}\rightarrow S'_{\epsilon}\rightrightarrows S'_{\epsilon}\otimes_{S_{\epsilon}}S'_{\epsilon}\mathrel{\substack{\textstyle\rightarrow\\[-0.4ex]
                      \textstyle\rightarrow \\[-0.4ex]
                      \textstyle\rightarrow}}S'_{\epsilon}\otimes_{S_{\epsilon}}S'_{\epsilon}\otimes_{S_{\epsilon}}S'_{\epsilon}\mathrel{\substack{\textstyle\rightarrow\\[-0.4ex]
                      \textstyle\rightarrow \\[-0.4ex]
                      \textstyle\rightarrow\\[-0.4ex]
                      \textstyle\rightarrow}}\cdots
\end{equation*}
is acyclic because $W^{\dagger}(S'_{\epsilon})$ is finite \'{e}tale and faithfully flat over $W^{\dagger}(S_{\epsilon})$.

\item $\{U_{\epsilon,i}\}$ is a covering by rational subdomains of $\mathrm{Spa}\,S_{\epsilon}$. Then by using the same reduction argument as in \cite[Chapter 8.2]{BGR84}, it suffices to consider Laurent coverings and then by an induction argument the case that $\mathrm{Spa}\,S_{\epsilon}$ is covered by $\mathrm{Spa}\,S_{\epsilon}\langle w^{\epsilon} f\rangle$ and $\mathrm{Spa}\,S_{\epsilon}\langle w^{\epsilon} f^{-1}\rangle$ for some $f\in S_{\epsilon}$. In this case \eqref{simplicial} becomes 
\begin{align}\label{simplicial new}
0\rightarrow\varinjlim_{\epsilon\rightarrow 0}W^{\dagger}(S_{\epsilon}))\rightarrow
& \varinjlim_{\epsilon}(W^{\dagger}(S_{\epsilon}\langle (w^{\epsilon} f)^{1/p^{\infty}}\rangle\times W^{\dagger}(S_{\epsilon}\langle (w^{\epsilon} f^{-1})^{1/p^{\infty}}\rangle) \\
& \rightarrow\varinjlim_{\epsilon}W^{\dagger}(S_{\epsilon}\langle (w^{\epsilon} f)^{1/p^{\infty}},(w^{\epsilon}f^{-1})^{1/p^{\infty}}\rangle)\rightarrow 0\,.\nonumber
\end{align}
\end{enumerate}
Let $S^{\dagger}\langle f^{1/p^{\infty}}\rangle=\varinjlim_{\epsilon}S_{\epsilon}\langle (w^{\epsilon}f)^{1/p^{\infty}}\rangle$ be the corresponding perfectoid dagger algebra. Then 
\begin{align*}
W^{\dagger}(S_{\epsilon}\langle (w^{\epsilon}f)^{1/p^{\infty}}\rangle)=\{z=(z_{0},z_{1},\ldots)
& \in W(S^{\dagger}\langle f^{1/p^{\infty}}\rangle)\,|\,\forall i:\, z_{i}\in S_{\epsilon}\langle (w^{\epsilon}f)^{1/p^{\infty}} \\
& \text{ and }\varinjlim_{i}\left(\frac{\gamma_{S_{\epsilon}\langle (w^{\epsilon f)^{1/p^{\infty}}}}(z_{i})}{p^{i}}+i\right)=\infty\}\,.
\end{align*}
Likewise, $W^{\dagger}(S_{\epsilon}\langle (w^{\epsilon}f^{-1})^{1/p^{\infty}}\rangle)$ and $W^{\dagger}(S_{\epsilon}\langle (w^{\epsilon}f)^{1/p^{\infty}},(w^{\epsilon}f^{-1})^{1/p^{\infty}}\rangle)$ are defined.

Let $\widehat{S}\langle f^{1/{p^{\infty}}}\rangle$ be the $w$-adic completion of $S^{\dagger}\langle f^{1/{p^{\infty}}}\rangle$. Under the isomorphism
\begin{equation*}
W(\widehat{S}\langle f^{1/{p^{\infty}}}\rangle)=W(\widehat{S})\langle [f]^{1/p^{\infty}}\rangle
\end{equation*}
(where $\langle\,\cdot\,\rangle$ denotes $p$-$d$-adic completion), let 
\begin{equation*}
R(f):=W^{\dagger}(S_{\epsilon})\langle [w^{\epsilon}f]^{1/p^{\infty}}\rangle
\end{equation*}
be the image of $W^{\dagger}(S_{\epsilon}\langle [w^{\epsilon}f]^{1/p^{\infty}}\rangle)$. Likewise, for a variable $\eta$, under the isomorphism 
\begin{equation*}
W(\widehat{S}\langle \eta^{1/{p^{\infty}}}\rangle)=W(\widehat{S})\langle [\eta]^{1/p^{\infty}}\rangle
\end{equation*}
let 
\begin{equation*}
R(\eta):=W^{\dagger}(S_{\epsilon})\langle [w^{\epsilon}\eta]^{1/p^{\infty}}\rangle
\end{equation*}
be the image of $W^{\dagger}(S_{\epsilon}\langle [w^{\epsilon}\eta]^{1/p^{\infty}}\rangle)$. Similarly for variables $\xi,\eta$ 
\begin{equation*}
R(\xi,\eta):=W^{\dagger}(S_{\epsilon})\langle [w^{\epsilon}\xi]^{1/p^{\infty}},[w^{\epsilon}\eta]^{1/p^{\infty}}\rangle
\end{equation*}
and
\begin{equation*}
R(\xi,\xi^{-1}):=W^{\dagger}(S_{\epsilon})\langle [w^{\epsilon}\xi]^{1/p^{\infty}},[w^{\epsilon}\xi^{-1}]^{1/p^{\infty}}\rangle
\end{equation*}
are defined.

Now consider the following commutative diagram
\begin{equation}
\label{new diagram}
\begin{tikzpicture}[descr/.style={fill=white,inner sep=1.5pt}]
        \matrix (m) [
            matrix of math nodes,
            row sep=2.5em,
            column sep=2.5em,
            text height=1.5ex, text depth=0.25ex
        ]
        { \ & \ & 0 & 0 & \ \\
        \ & \ & I(\xi)\times I(\eta) & I(\xi)_{R(\xi,\xi^{-1})} & \ \\
        0 & W^{\dagger}(S_{\epsilon}) & R(\xi)\times R(\eta) & R(\xi,\xi^{-1}) & 0 \\
        0 & W^{\dagger}(S_{\epsilon}) & R(f)\times R(f^{-1}) & R(f,f^{-1}) & 0 \\
        \ & \ & 0 & 0 & \ \\};

        \path[overlay,->, font=\scriptsize]         
        (m-2-3) edge node[above]{$\lambda'$}(m-2-4)
        (m-3-1) edge (m-3-2)
        (m-3-2) edge (m-3-3)        
        (m-3-3) edge node[above]{$\lambda$}(m-3-4)
        (m-3-4) edge (m-3-5) 
        (m-4-1) edge (m-4-2)
        (m-4-2) edge (m-4-3)        
        (m-4-3) edge (m-4-4)
        (m-4-4) edge (m-4-5) 
        (m-1-3) edge (m-2-3)         
        (m-1-4) edge (m-2-4) 
        (m-2-3) edge (m-3-3)
        (m-2-4) edge (m-3-4)
        (m-3-2) edge node[right]{$=$} (m-4-2) 
        (m-3-3) edge (m-4-3)
        (m-3-4) edge (m-4-4) 
        (m-4-3) edge (m-5-3)
        (m-4-4) edge (m-5-4)        
        ;
                                                
\end{tikzpicture}
\end{equation}
$\lambda$ is the map $(h_{1}(w^{\epsilon}\xi),h_{2}(w^{\epsilon}\eta))\mapsto(h_{1}(w^{\epsilon}\xi)-h_{2}(w^{\epsilon}\xi^{-1}))$, $\lambda'$ is induced by $\lambda$, the vertical maps are the canonical ones given by $[w^{\epsilon}\xi]^{k}\mapsto[w^{\epsilon}f]^{k}$ and $[w^{\epsilon}\eta]^{k}\mapsto[w^{\epsilon}f^{-1}]^{k}$ for $k\in\mathbb{Z}_{\geq 0}[1/p]$. $I(\xi)$ denotes the $(p,d)$-completed ideal in $R(\xi)$ generated by $[w^{\epsilon}\xi]^{k}-[w^{\epsilon}f]^{k}$, $k\in\mathbb{Z}_{\geq 0}[1/p]$ and, likewise, $I(\eta)$ is the $(p,d)$-completed ideal in $R(\eta)$ generated by $1-[f]^{k}[\eta]^{k}$, $k\in\mathbb{Z}_{\geq 0}[1/p]$. It is then clear that the second column is exact.

Now
\begin{equation*}
R(f,f^{-1})=R(\xi,\eta)/\langle ([w^{\epsilon}\xi]^{k})-[w^{\epsilon}f]^{k}),(1-[f]^{k}[\eta]^{k})\rangle_{k\in\mathbb{Z}_{\geq 0}[1/p]}
\end{equation*}
and since the ideal
\begin{equation*}
\langle ([w^{\epsilon}\xi]^{k})-[w^{\epsilon}f]^{k}),(1-[f]^{k}[\eta]^{k})\rangle_{k\in\mathbb{Z}_{\geq 0}[1/p]}
\end{equation*}
coincides with the ideal 
\begin{equation*}
\langle ([w^{\epsilon}\xi]^{k})-[w^{\epsilon}f]^{k}),(1-[\xi]^{k}[\eta]^{k})\rangle_{k\in\mathbb{Z}_{\geq 0}[1/p]}
\end{equation*}
we obtain
\begin{equation*}
R(f,f^{-1})=R(\xi,\xi^{-1})/I(\xi)_{R(\xi,\xi^{-1})}
\end{equation*}
and hence the third column is exact too. (Note that $I(\xi)_{R(\xi,\xi^{-1})}$ is the $(p,d)$-completed ideal in $R(\xi,\xi^{-1})$ generated by $[w^{\epsilon}\xi]^{k}-[w^{\epsilon}f]^{k}$ for $k\in\mathbb{Z}_{\geq 0}[1/p]$.

The equations
\begin{equation}\label{plus}
R(\xi,\xi^{-1})=R(\xi)+\langle w^{\epsilon},\xi^{-1}\rangle^{1/p^{\infty}}R(\xi^{-1})
\end{equation}
and
\begin{equation*}
I(\xi)_{R(\xi,\xi^{-1})}=I(\xi)_{R(\xi)}+\langle 1-[f]^{k}[\xi^{-1}]^{k}\rangle_{k\in\mathbb{Z}_{\geq 0}[1/p]}R(\xi^{-1})
\end{equation*}
show the surjectivity of $\lambda$ and $\lambda'$, and the exactness of the first row in \eqref{new diagram}. The second row in \eqref{new diagram} is easily seen to be exact too. 

We already know that after taking $\varinjlim_{\epsilon}$ the last row is exact in the middle. But a simple diagram chase in \eqref{new diagram} yields that the last row is already exact before taking $\varinjlim_{\epsilon}$. Hence the complex \eqref{simplicial} is acyclic. Hence $H^{i}_{\mathrm{\acute{e}t}}(X^{\flat},W^{\dagger}(\mathcal{O}^{\flat}_{X}))=H^{i}_{\mathrm{pro\acute{e}t}}(\mathcal{U},W^{\dagger}(\mathcal{O}^{\flat}_{X}))=0$ for $i>0$.

To obtain an almost vanishing result for the cohomology of the sheaves $W^{\dagger}(\mathcal{O}_{X_{\epsilon}}^{+\flat})$ we need an appropriate definition of almost finite \'{e}taleness for $W(\mathcal{O}^{\flat})$-modules.

\begin{defn}
\begin{itemize}
\item[a)] Let $R$ be a $W(\mathcal{O}^{\flat})$-algebra and $N$ an $R$-module. Then $N$ is uniformly finitely generated if there exists some integer $n$ such that for all $\epsilon\in W(\mathfrak{m}^{\flat})$ there exists an $R$-module $N_{\epsilon}$, finitely generated by $n$ elements, and a map $f_{\epsilon}:N_{\epsilon}\rightarrow N$ such that the kernel and cokernel of $f_{\epsilon}$ are annihilated by $\epsilon$.

\item[b)] Let $R$ be a $W(\mathcal{O}^{\flat})$-algebra and $N$ an $R$-module. Then $N$ is almost projective over $R$ if for all $R$-modules $X$ and all $i>0$, $\mathrm{Ext}_{R}^{i}(N,X)$ is almost zero (i.e. annihilated by $W(\mathfrak{m}^{\flat})$).

\item[c)] A morphism of $W(\mathcal{O}^{\flat})$-modules $A\rightarrow B$ is almost unramified if $A\otimes_{W(\mathcal{O}^{\flat})}W(K^{\flat})\rightarrow B\otimes_{W(\mathcal{O}^{\flat})}W(K^{\flat})$ is unramified and the corresponding idempotent $e\in(B\otimes_{W(\mathcal{O}^{\flat})}W(K^{\flat}))\otimes_{A\otimes_{W(\mathcal{O}^{\flat})}W(K^{\flat})}(B\otimes_{W(\mathcal{O}^{\flat})}W(K^{\flat}))$ defines an almost element in $B\otimes_{A}B$ : that is $e$ lies in $\mathrm{Hom}(W(\mathfrak{m}^{\flat}),B\otimes_{A}B)$ under the map $\epsilon\mapsto\epsilon e$.

\item[d)] A morphism $A\rightarrow B$ of $W(\mathcal{O}^{\flat})$-modules is almost finite \'{e}tale if it is almost unramified, almost projective and uniformly almost finitely presented.
\end{itemize}
\end{defn} 

\begin{prop}\label{almost finite etale}
The assumptions are as above, so $S'/S$ is finite \'{e}tale of degree $n$, $S'^{\circ}/S^{\circ}$ is almost finite \'{e}tale. Then $W^{\dagger}(S'^{\circ})$ is almost finite \'{e}tale over $W^{\dagger}(S^{\circ})$.
\end{prop}
\begin{proof}
Let $\overline{e}\in S'\otimes_{S}S'$ be the idempotent showing that $S'$ is unramified over $S$. Then for all $\overline{\epsilon}\in\mathfrak{m}^{\flat}$ $\overline{\epsilon}\cdot\overline{e}\in S'^{\circ}\otimes_{S^{\circ}}S'^{\circ}$ (see \cite[Proposition 5.23]{Sch12}). Then $e:=[\overline{e}]\in W^{\dagger}(S'\otimes_{S}S')=W^{\dagger}(S')\otimes_{W^{\dagger}(S)}W^{\dagger}(S')$ is an idempotent showing unramifiedness of $W^{\dagger}(S')$ over $W^{\dagger}(S)$. Then for all $\epsilon\in W(\mathfrak{m}^{\flat})$, $\epsilon=(\epsilon_{0},\epsilon_{1},\epsilon_{2},\ldots)$ we have $\epsilon\cdot e=\epsilon[\overline{e}]=(\epsilon_{0}\overline{e},\epsilon_{1}\overline{e},\epsilon_{2}\overline{e},\ldots)$ is an element in $W^{\dagger}(S'^{\circ}\otimes_{S^{\circ}}S'^{\circ})$. Write 
\begin{equation*}
e=\left[\displaystyle\sum_{\stackrel{i=1}{j=1}}^{n}\lambda_{ij}(x_{i}\otimes x_{j})\right]
\end{equation*}
for an $S$-basis $x_{1},\ldots, x_{n}$ of $S'$. Choose $u\in\mathfrak{m}^{\flat}$ such that $ux_{j}=y_{j}\in S'^{\circ}$ for all $j$. Then 
\begin{equation*}
e=\left[\displaystyle\sum_{\stackrel{i=1}{j=1}}^{n}\lambda_{ij}(x'_{i}\otimes y_{j})\right]
\end{equation*}
with $x'_{i}=x_{i}/u$. Using \cite[Lemma A.9]{LZ04} we can also write - for each $z\in S'\otimes_{S}S'$ and $l\in\mathbb{N}$:
\begin{equation*}
z=\sum_{i=1}^{n}\xi_{ij}^{(l)}(x_{i}\otimes x_{j}^{p^{l}})
\end{equation*}
and hence
\begin{equation*}
z=\sum_{i=1}^{n}\xi_{ij}^{(l)}x'_{i,l}\otimes y_{j}^{p^{l}}
\end{equation*}
for $x'_{i,l}=x_{i}/u^{p^{l}}$. We have $y_{j}^{p^{l}}\in S'^{\circ}$ for all $l$, by construction then, in $W^{\dagger}(S')\otimes_{W^{\dagger}(S)}W^{\dagger}(S')$, we can write $e$ as follows:
\begin{equation}\label{how to write e}
e=\sum_{i,j}[\lambda_{ij}(x_{i}')]\otimes[y_{j}]-\sum_{i,j}[0,m_{ij}^{(1)},m_{ij}^{(2)},\ldots]
\end{equation}
where $m_{ij}^{(l)}=\xi_{ij}^{(l)}[x_{i,l}'\otimes y_{j}^{p^{l}}]$ with uniquely determined $\xi_{ij}^{(l)}\in S$. Hence we can write 
\begin{equation}\label{how to write e 2}
e=\sum_{i,j}[\lambda_{ij}(x_{i}')]\otimes[y_{j}]-\sum_{i,j}[0,\xi_{ij}^{(1)}x'_{i,1},\xi_{ij}^{(2)}x'_{i,2},\ldots]\otimes[y_{j}]
\end{equation}
as an element in $W^{\dagger}(S')\otimes_{W^{\dagger}(S)}W^{\dagger}(S')$.

For $\epsilon_{0}\in\mathfrak{m}^{\flat}$, $[\epsilon_{0}]e\in W^{\dagger}(S'^{\circ}\otimes_{S^{\circ}}S'^{\circ})$ as above. The uniqueness of the representation \eqref{how to write e 2} of $e$ implies that $\epsilon_{0}\lambda_{ij}x'_{i}\in S'^{\circ}$ for all $i,j$ and likewise $[\epsilon_{0}][0,\xi_{ij}^{(1)}x'_{i,1},\xi_{ij}^{(2)}x'_{i,2},\ldots]\in W^{\dagger}(S'^{\circ})$ and therefore $[\epsilon_{0}]e\in W^{\dagger}(S'^{\circ})\otimes_{W^{\dagger}(S^{\circ})}W^{\dagger}(S'^{\circ})$. We also see that $V^{s}[\epsilon_{s}]e=p^{s}[\epsilon_{s}]^{1/p^{s}}e\in W^{\dagger}(S'^{\circ})\otimes_{W^{\dagger}(S^{\circ})}W^{\dagger}(S'^{\circ})$ for all $\epsilon_{s}\in\mathfrak{m}^{\flat}$, and therefore for all $\epsilon\in W(\mathfrak{m}^{\flat})$ we have $\epsilon\cdot e\in W^{\dagger}(S'^{\circ})\otimes_{W^{\dagger}(S^{\circ})}W^{\dagger}(S'^{\circ})$ and therefore $W^{\dagger}(S'^{\circ})$ is almost unramified over $W^{\dagger}(S^{\circ})$. 

Let $\epsilon\in W(\mathfrak{m}^{\flat})$ and $\epsilon\cdot e=\sum_{i=1}^{r}a_{i}\otimes b_{i}$ for $a_{i},b_{i}\in W^{\dagger}(S'^{\circ})$ (so in the following $\epsilon$ is fixed). Consider the map 
\begin{equation*}
W^{\dagger}(S'^{\circ})\rightarrow W^{\dagger}(S^{\circ})^{n}, s\mapsto(\mathrm{Tr}_{W^{\dagger}(S')/W^{\dagger}(S)}(s,b_{1}),\ldots, \mathrm{Tr}_{W^{\dagger}(S')/W^{\dagger}(S)}(s,b_{n}))
\end{equation*}
and the map $W^{\dagger}(S^{\circ})^{n}\rightarrow W^{\dagger}(S'^{\circ})$ defined by $(r_{1},\ldots, r_{n})\mapsto\sum_{i=1}^{n}a_{i}r_{i}$. As in the proof of \cite[Proposition 5.23]{Sch12} one easily checks that the composite map $W^{\dagger}(S'^{\circ})\rightarrow W^{\dagger}(S^{\circ})^{n}\rightarrow W^{\dagger}(S'^{\circ})$ is multiplication by $\epsilon$. This shows that $W^{\dagger}(S'^{\circ})$ is a uniformly almost finitely presented almost projective $W^{\dagger}(S^{\circ})$-module and proves Proposition \ref{almost finite etale}.
\end{proof}

Now consider the sequence \eqref{simplicial} at the ``integral level'', namely the total complex associated to the simplicial complex using the presheaves $\mathcal{O}^{+\flat}_{X_{\epsilon}}$, resp. $W^{\dagger}(\mathcal{O}^{+\flat}_{X_{\epsilon}})$. We have
\begin{lemma}\label{integral lemma}
The total complex associated to the simplicial complex 
\begin{equation}\label{simplicial 2}
0\rightarrow\varinjlim_{\epsilon}W^{\dagger}(\mathcal{O}_{X_{\epsilon}}^{+\flat})(X_{\epsilon}^{\flat})\rightarrow\prod_{i}\varinjlim_{\epsilon}W^{\dagger}(\mathcal{O}_{X_{\epsilon}}^{+\flat})(U_{i})\rightrightarrows\prod_{i,j}\varinjlim_{\epsilon}W^{\dagger}(\mathcal{O}_{X_{\epsilon}}^{+\flat})(U_{i}\times U_{j})\mathrel{\substack{\textstyle\rightarrow\\[-0.4ex]
                      \textstyle\rightarrow \\[-0.4ex]
                      \textstyle\rightarrow}}\cdots
\end{equation}
has cohomology killed by $W(\mathfrak{m}^{\flat})$.
\end{lemma}
\begin{proof}
Again we consider the two cases 
\begin{enumerate}[i)]
\item $U_{\epsilon}=\mathrm{Spa}\,S'_{\epsilon}$ finite \'{e}tale faithfully flat over $\mathrm{Spec}\,\S_{\epsilon}$. Since $W^{\dagger}(S'_{\epsilon})$ is finite \'{e}tale over $W^{\dagger}(S_{\epsilon})$, the canonical map
\begin{equation*}
W^{\dagger}(S'_{\epsilon})\otimes_{W^{\dagger}(S_{\epsilon})}W^{\dagger}(S'_{\epsilon})\rightarrow W^{\dagger}(S'_{\epsilon}\otimes_{S_{\epsilon}}S'_{\epsilon})
\end{equation*}
defined by
\begin{equation*}
(\sum_{i\geq 0}p^{i}[x_{i}]^{1/p^{i}})\otimes (\sum_{j\geq 0}p^{j}[x_{j}]^{1/p^{j}})\mapsto \sum_{i,j\geq 0}p^{i+j}[x_{i}^{1/p^{i}}\otimes x_{j}^{1/p^{j}}]
\end{equation*}
for $x_{i},x_{j}\in S'_{\epsilon}$, is an isomorphism. Proposition \ref{almost finite etale} implies that the same map, defined at integral level,
\begin{equation*}
W^{\dagger}(S'^{\circ}_{\epsilon})\otimes_{W^{\dagger}(S^{\circ}_{\epsilon})}W^{\dagger}(S'^{\circ}_{\epsilon})\rightarrow W^{\dagger}(S'^{\circ}_{\epsilon}\otimes_{S^{\circ}_{\epsilon}}S'^{\circ}_{\epsilon})
\end{equation*}
has kernel and cokernel killed by $W(\mathfrak{m}^{\flat})$. Moreover, Proposition \ref{almost finite etale} implies that the complex \eqref{simplicial 2} associated to the acyclic complex 
\begin{equation*}
0\rightarrow S_{\epsilon}\rightarrow S'_{\epsilon}\rightrightarrows S'_{\epsilon}\otimes_{S_{\epsilon}}S'_{\epsilon}\mathrel{\substack{\textstyle\rightarrow\\[-0.4ex]
                      \textstyle\rightarrow \\[-0.4ex]
                      \textstyle\rightarrow}}\cdots\,,
\end{equation*}
namely
\begin{equation*}
0\rightarrow W^{\dagger}(S^{\circ}_{\epsilon})\rightarrow W^{\dagger}(S'^{\circ}_{\epsilon})\rightrightarrows W^{\dagger}(S'^{\circ}_{\epsilon}\otimes_{S^{\circ}_{\epsilon}}S'^{\circ}_{\epsilon})\mathrel{\substack{\textstyle\rightarrow\\[-0.4ex]
                      \textstyle\rightarrow \\[-0.4ex]
                      \textstyle\rightarrow}}\cdots
\end{equation*}
has cohomology killed by $W(\mathfrak{m}^{\flat})$. Indeed, for $\kappa\in W(\mathfrak{m}^{\flat})$, let $\lambda_{\kappa}:W^{\dagger}(S^{\circ}_{\epsilon})^{n}\rightarrow W^{\dagger}(S'^{\circ}_{\epsilon})$ be the map considered in the proof of Proposition \ref{almost finite etale} such that the cohomology of the cone of $\lambda_{\kappa}$ is killed by $\kappa$. The corresponding simplicial complex for $M:=W^{\dagger}(S^{\circ}_{\epsilon})^{n}$
\begin{equation*}
0\rightarrow W^{\dagger}(S^{\circ}_{\epsilon})\rightarrow M\rightrightarrows M\otimes_{W^{\dagger}(S^{\circ}_{\epsilon})}M\mathrel{\substack{\textstyle\rightarrow\\[-0.4ex]
                      \textstyle\rightarrow \\[-0.4ex]
                      \textstyle\rightarrow}}\cdots
\end{equation*}
is obviously acyclic, hence the cohomology of the complex \eqref{simplicial 2} is killed by $\kappa$. Since this holds for all $\kappa$, the claim follows.

\item Now let $\{U_{\epsilon,i}\}$ be a covering by rational subdomains of $\mathrm{Spa}\,S_{\epsilon}$. Using again the same reduction argument as in \cite[\S8.2]{BGR84} it suffices to consider Laurent coverings and then by induction the case $\mathrm{Spa}\,S_{\epsilon}=\mathrm{Spa}\,S_{\epsilon}\langle w^{\epsilon}f\rangle\cup\mathrm{Spa}\,S_{\epsilon}\langle w^{\epsilon}f^{-1}\rangle$ for some $f\in S_{\epsilon}$.

In the limit we have the covering $\mathrm{Spa}\,S^{\dagger}=\mathrm{Spa}\,S^{\dagger}\langle f\rangle\cup\mathrm{Spa}\,S^{\dagger}\langle 1/f\rangle$ where $S^{\dagger}$ is a perfectoid dagger algebra. Since we consider integral elements we can use the argument in \cite[Lemma 6.4]{Sch12} to assume that $f\in S^{\dagger,\circ}$ at the expense of writing $\mathrm{Spa}\,S^{\dagger}=\mathrm{Spa}\,S^{\dagger}\langle f\rangle\cup\mathrm{Spa}\,S^{\dagger}\langle w^{N}/f\rangle$ for some $N$. The Laurent covering does not change. Then, to show Lemma \ref{integral lemma}, we need to show that the sequence
\begin{align}
\label{sequence in case ii}
0\rightarrow\varinjlim_{\epsilon\rightarrow 0}W^{\dagger}(S^{\circ}_{\epsilon})\rightarrow
& \varinjlim_{\epsilon\rightarrow 0}W^{\dagger}(S^{\circ}_{\epsilon}\langle (w^{\epsilon}f)^{1/p^{\infty}}\rangle)\times W^{\dagger}(S^{\circ}_{\epsilon}\langle (w^{\epsilon}w^{N}/f)^{1/p^{\infty}}\rangle) \\
& \rightarrow\varinjlim_{\epsilon\rightarrow 0}W^{\dagger}(S^{\circ}_{\epsilon}\langle (w^{\epsilon}f)^{1/p^{\infty}},(w^{\epsilon}w^{N}/f)^{1/p^{\infty}}\rangle)\rightarrow 0 \nonumber
\end{align}
has cohomology killed by $W(\mathfrak{m}^{\flat})$.

Define, in analogy to the `generic' case in diagram \eqref{new diagram}, 
\begin{align*}
& R^{\circ}(f)=W^{\dagger}(S^{\circ}_{\epsilon}\langle (w^{\epsilon}f)^{1/p^{\infty}}\rangle) \\
& R^{\circ}(\eta)=W^{\dagger}(S^{\circ}_{\epsilon}\langle (w^{\epsilon}\eta)^{1/p^{\infty}}\rangle)\text{ for a variable }\eta\text{ and likewise for }\xi \\
& R^{\circ}(w^{N}/f)=W^{\dagger}(S^{\circ}_{\epsilon}\langle (w^{\epsilon}w^{N}f^{-1})^{1/p^{\infty}}\rangle) \\
& R^{\circ}(\xi,w^{N}/\xi)=W^{\dagger}(S^{\circ}_{\epsilon}\langle (w^{\epsilon}\xi)^{1/p^{\infty}}, (w^{\epsilon}w^{N}\xi^{-1})^{1/p^{\infty}}\rangle)\text{ and likewise for }R^{\circ}(f,w^{N}/f)\,.
\end{align*}
In analogy to diagram \eqref{new diagram} we consider the following commutative diagram at ``integral'' level:
\begin{equation}
\label{integral diagram}
\begin{tikzpicture}[descr/.style={fill=white,inner sep=1.5pt}]
        \matrix (m) [
            matrix of math nodes,
            row sep=2.5em,
            column sep=2.5em,
            text height=1.5ex, text depth=0.25ex
        ]
        { \ & \ & 0 & 0 & \ \\
        \ & \ & I^{\circ}(\xi)\times I^{\circ}(\eta) & I^{\circ}(\xi)_{R^{\circ}(\xi,w^{N}/\xi)} & \ \\
        0 & W^{\dagger}(S^{\circ}_{\epsilon}) & R^{\circ}(\xi)\times R^{\circ}(\eta) & R^{\circ}(\xi,w^{N}/\xi) & 0 \\
        0 & W^{\dagger}(S^{\circ}_{\epsilon}) & R^{\circ}(f)\times R^{\circ}(w^{N}/f) & R^{\circ}(f,w^{N}/f) & 0 \\
        \ & \ & 0 & 0 & \ \\};

        \path[overlay,->, font=\scriptsize]         
        (m-2-3) edge node[above]{$\lambda'$}(m-2-4)
        (m-3-1) edge (m-3-2)
        (m-3-2) edge (m-3-3)        
        (m-3-3) edge node[above]{$\lambda$}(m-3-4)
        (m-3-4) edge (m-3-5) 
        (m-4-1) edge (m-4-2)
        (m-4-2) edge (m-4-3)        
        (m-4-3) edge (m-4-4)
        (m-4-4) edge (m-4-5) 
        (m-1-3) edge (m-2-3)         
        (m-1-4) edge (m-2-4) 
        (m-2-3) edge (m-3-3)
        (m-2-4) edge (m-3-4)
        (m-3-2) edge node[right]{$=$} (m-4-2) 
        (m-3-3) edge (m-4-3)
        (m-3-4) edge (m-4-4) 
        (m-4-3) edge (m-5-3)
        (m-4-4) edge (m-5-4)        
        ;
                                                
\end{tikzpicture}
\end{equation}
$\lambda$ is the map 
\begin{equation*}
(\sum_{i}p^{i}[h_{1}^{(i)}(w^{\epsilon}\xi)]^{1/p^{i}},\sum_{i}p^{i}[h_{2}^{(i)}(w^{\epsilon}\eta)]^{1/p^{i}})\mapsto\sum_{i}p^{i}[h_{1}^{(i)}(w^{\epsilon}\xi)]-h_{2}^{(i)}(w^{\epsilon}w^{N}\eta^{-1})]^{1/p^{i}}
\end{equation*}
\begin{equation*}
I^{\circ}(\xi)=\{\sum_{i}p^{i}[w^{\epsilon}\xi-w^{\epsilon}f]^{r_{i}}[\theta_{i}]\,|\, r_{i}\in\mathbb{Z}_{\geq 0}[1/p],\theta_{i}\in S^{\circ}_{\epsilon}\langle w^{\epsilon}\xi\rangle^{1/p^{\infty}}\}
\end{equation*}
where we only consider elements that still lie in $W^{\dagger}(S^{\circ}_{\epsilon}\langle w^{\epsilon}\xi\rangle ^{1/p^{\infty}})$, i.e. which are overconvergent.
\begin{equation*}
I^{\circ}(\eta)=\{\sum_{i}p^{i}[w^{\epsilon}w^{N}-w^{\epsilon}\eta f]^{r_{i}}[\theta_{i}]\,|\, r_{i}\in\mathbb{Z}_{\geq 0}[1/p],\theta_{i}\in S^{\circ}_{\epsilon}\langle w^{\epsilon}\eta\rangle^{1/p^{\infty}}\}
\end{equation*}
again with the condition that these elements are overconvergent, hence lie in $W^{\dagger}(S_{\epsilon}^{\circ}(w^{\epsilon}\eta)^{1/p^{\infty}})$. $I^{\circ}(\xi)_{R^{\circ}(\xi,w^{N}/\xi)}$ is defined analogously, where we assume that $r_{i}\in S^{\circ}_{\epsilon}((w^{\epsilon}\xi)^{1/p^{\infty}},(w^{\epsilon}w^{N}\xi^{-1})^{1/p^{\infty}})$. The map $\lambda'$ in diagram \eqref{integral diagram} is then induced by $\lambda$. The vertical maps are the obvious ones given by $([h_{1}[w^{\epsilon}\xi],[h_{2}(w^{\epsilon}\eta)])\mapsto (([h_{1}[w^{\epsilon}f],[h_{2}(w^{\epsilon}w^{N}/f)])$. Since 
\begin{equation*}
[w^{\epsilon}\xi-w^{\epsilon}f]^{k}\cdot[w^{N}\xi^{-1}]^{k}=[w^{\epsilon}w^{N}-w^{\epsilon}w^{N}f\xi^{-1}]^{k}
\end{equation*}
for $k\in\mathbb{Z}_{\geq 0}[1/p]$, we see that $\lambda'$ is bijective. As in the generic case \eqref{new diagram} it is easy to see that the middle row in \eqref{integral diagram} is exact. Now consider, in particular, the map $\mu:R^{\circ}(\eta)\rightarrow R^{\circ}(w^{N}/f)$ in the second column in \eqref{integral diagram}. Then, modulo $p$, it follows from the proof of \cite[Lemma 6.4]{Sch12} that
\begin{equation*}
S^{\circ}_{\epsilon}\langle (w^{\epsilon}\eta)^{1/p^{\infty}}\rangle/(I^{\circ}(\eta)\mod p)\rightarrow S^{\circ}_{\epsilon}\langle (w^{\epsilon}w^{N}f^{-1})^{1/p^{\infty}}\rangle
\end{equation*}
is an almost isomorphism, the kernel is killed by $\mathfrak{m}^{\flat}$. This implies that the map
\begin{equation*}
\mu:R^{\circ}(\eta)/I^{\circ}(\eta)\rightarrow R^{\circ}(w^{N}/f)
\end{equation*}
is an almost isomorphism with respect to the ideal $[\mathfrak{m}^{\flat}]=\langle [x]\,|\,x\in\mathfrak{m}^{\flat}\rangle$, hence $[\mathfrak{m}^{\flat}]\ker\mu\subset I^{\circ}(\eta)$.

Now consider the corresponding rings/ideals for the full ring of Witt vectors: $\widehat{R}^{\circ}(f)=W(S^{\circ}_{\epsilon}\langle(w^{\epsilon}f)^{1/p^{\infty}}\rangle)$, likewise $\widehat{R}^{\circ}(\eta)$, $\widehat{R}^{\circ}(w^{N}/f)$, $\widehat{R}^{\circ}(\xi,w^{N}/\xi)$, $\widehat{R}^{\circ}(f, w^{N}/f)$, obtained by taking $p$-adic completions, also $\widehat{I}^{\circ}(\eta)$, the $p$-adic completion of $I^{\circ}(\eta)$. We have the induced map
\begin{equation*}
\widehat{\mu}:\widehat{R}^{\circ}(\eta)\rightarrow\widehat{R}^{\circ}(w^{N}/f)\,.
\end{equation*}
Since $\widehat{I}^{\circ}(\eta)$ is $p$-adically complete and $W(\mathfrak{m}^{\flat})$ is the $p$-adic completion of $[\mathfrak{m}^{\flat}]$, we conclude that $W(\mathfrak{m}^{\flat})\ker\widehat{\mu}\subset\widehat{I}^{\circ}(\eta)$, so $W(\mathfrak{m}^{\flat})$ kills $\ker\widehat{\mu}/\widehat{I}^{\circ}(\eta)$.

On the other hand, returning to overconvergent elements, we see that multiplication by $\lambda\in W(\mathfrak{m}^{\flat})$ maps $\ker\mu$ to overconvergent elements in $R^{\circ}(\eta)$, hence to $I^{\circ}(\eta)$, hence $W(\mathfrak{m}^{\flat})$ kills $\ker\mu/I^{\circ}(\eta)$. A diagram chase in \eqref{integral diagram} yields then that the cohomology of the bottom row in \eqref{integral diagram} is killed by $W(\mathfrak{m}^{\flat})$. 
\end{enumerate}
This finishes the proof of Lemma \ref{integral lemma}.
\end{proof}
This finishes the proof of Proposition \ref{almost zero}.
\end{proof}

\begin{cor}\label{almost quasi-isom}
We have an almost quasi-isomorphism
\begin{equation*}
R\Gamma_{\mathrm{gp}}(\mathbb{Z}_{p}(1)^{d},W^{\dagger}(R_{\infty}^{\flat}))\simeq R\Gamma_{\mathrm{pro\acute{e}t}}(U,W^{\dagger}(\mathcal{O}^{+\flat}_{X}))\,.
\end{equation*}
(This means that the cohomology of the cone is killed by $W(\mathfrak{m}^{\flat})$).
\end{cor}

The d\'{e}calage functor $L\eta_{\mu}$ applied to $R\Gamma_{\mathrm{pro\acute{e}t}}(U,W^{\dagger}(\mathcal{O}^{+\flat}_{X}))$ yields, by definition, $A^{\dagger}\Omega_{\mathcal{X}/\mathcal{O}}^{\mathrm{pro\acute{e}t}}$. We have the following lemma:

\begin{lemma}\label{quasi-isom}
$L\eta_{\mu}$ transforms the almost quasi-isomorphism in Corollary \ref{almost quasi-isom} into a quasi-isomorphism.
\end{lemma}
\begin{proof}
Since $W(\mathfrak{m}^{\flat})\neq W(\mathfrak{m}^{\flat})^{2}$ we cannot apply the usual technology of almost mathematics. Instead we use the following proposition of Bhatt \cite[Lemma 6.14]{Bha18b}:

\begin{prop}\label{Bhatt prop}
Let $\mathfrak{m}$ be an ideal of a ring $A$ and $f$ a non-zero divisor in $\mathfrak{m}$. Let $\sigma:C\rightarrow D$ be a homomorphism of complexes of $A$-modules such that
\begin{enumerate}[i)]
\item the cone of $\sigma$ is killed by $\mathfrak{m}$.
\item all cohomology groups of $C\otimes^{L}A/fA$ contain no non-zero elements killed by $\mathfrak{m}^{2}$. 
\end{enumerate}
Then $L\eta_{f}C\rightarrow L\eta_{f}D$ is a quasi-isomorphism.
\end{prop}

To continue the proof of Lemma \ref{quasi-isom}, property i) in Proposition \ref{Bhatt prop} is satisfied since our map is an almost quasi-isomorphism. For ii) we use Morrow's arguments in his notes \cite[p. 42]{Mor16}. We need to show that the cohomology of $R\Gamma_{\mathrm{gp}}(\Gamma,W^{\dagger}((R^{\dagger}_{\infty})^{\flat}/\mu)))$ has no non-zero elements killed by $W(\mathfrak{m}^{\flat})^{2}$. Now, $R\Gamma_{\mathrm{gp}}(\Gamma,W^{\dagger}((R^{\dagger}_{\infty})^{\flat}/\mu)))$ is quasi-isomorphic to the weak $p$-adic completion of
\begin{equation*}
\bigoplus_{k_{1},\ldots,k_{d}\in\mathbb{Z}[1/p]}K_{A_{\inf}/\mu A_{\inf}}([\epsilon]^{k_{1}}-1,\ldots,[\epsilon]^{k_{d}}-1)
\end{equation*}
which is a dagger completion of Koszul complexes. By \cite[Lemma 7.10]{BMS18}, the cohomology of each of these complexes is the weak $p$-adic completion of a finite direct sum of copies of $A_{\inf}/\mu A_{\inf}([\epsilon]^{k}-1)$ and $A_{\inf}/([\epsilon]^{k}-1)A_{\inf}$ for varying $k\in\mathbb{Z}[1/p]$. It is shown in \cite[p. 42]{Mor16} that these $p$-torsion-free modules contain no non-zero elements killed by $W(\mathfrak{m}^{\flat})^{2}$.
\end{proof}

Now we are going to relate the group cohomology $L\eta_{\mu}R\Gamma_{\mathrm{gp}}(\Gamma,W^{\dagger}((R_{\infty})^{\flat})))$ to a dagger version of the $q$-de Rham complex. In analogy to the $\Gamma$-equivariant decomposition 
\begin{equation*}
A_{\inf}(\widehat{R}_{\infty})=A(\widehat{R})^{\square}\oplus A_{\inf}(\widehat{R}_{\infty})^{\mathrm{nonint}}
\end{equation*}
(see below \cite[Lemma 9.6]{BMS18}) we have the dagger version 
\begin{equation*}
A_{\inf}^{\dagger}(R_{\infty})=W^{\dagger}((R_{\infty})^{\flat})=A(R)^{\square}\oplus A^{\dagger}_{\inf}(R_{\infty})^{\mathrm{nonint}}\,.
\end{equation*}
Here $A(R)^{\square}$ is a lifting of $R$ over $A_{\inf}$, it is \'{e}tale over $A_{\inf}^{\dagger}(\underline{U}^{\pm 1})$. Following the proof of \cite[Lemma 9.6]{BMS18}, we conclude that $H^{i}_{\mathrm{cont}}(\Gamma,A_{\inf}^{\dagger}(R_{\infty})^{\mathrm{nonint}})$ is killed by $\mu$ and hence $L\eta_{\mu}R\Gamma_{\mathrm{cont}}(\Gamma,A_{\inf}^{\dagger}(R_{\infty})^{\mathrm{nonint}})$ vanishes.

On the other hand, we have the $q$-derivation
\begin{equation*}
\frac{\partial_{q}}{\partial_{q}\log(U_{i})}=\frac{\gamma_{i}-1}{[\epsilon]-1}
\end{equation*}
action on $A(R)^{\square}$ (see \cite[\S9.2]{BMS18}). It is well-known that we can compute group cohomology via Koszul complexes. The differentials in the complex 
\begin{equation*}
R\Gamma_{\mathrm{cont}}(\Gamma,A(R)^{\square})=K_{A(R)^{\square}}(\gamma_{1}-1,\ldots,\gamma_{d}-1)
\end{equation*}
are divisible by $\mu=[\epsilon]-1$ and we get by \cite[Lemma 7.9]{BMS18}  
\begin{equation*}
\eta_{\mu}R\Gamma_{\mathrm{cont}}(\Gamma,A(R)^{\square})=K_{A(R)^{\square}}\left(\frac{\gamma_{1}-1}{[\epsilon]-1},\ldots,\frac{\gamma_{d}-1}{[\epsilon]-1}\right)
\end{equation*}
and we define the $q$dagger de Rham complex of $A(R)^{\square}$ as 
\begin{equation*}
q\Omega^{\dagger\bullet}_{A(R)^{\square}/A_{\inf}}:=K_{A(R)^{\square}}\left(\frac{\gamma_{1}-1}{[\epsilon]-1},\ldots,\frac{\gamma_{d}-1}{[\epsilon]-1}\right)\,.
\end{equation*}
Hence, by definition, we have a quasi-isomorphism
\begin{equation*}
L\eta_{\mu}R\Gamma_{\mathrm{cont}}(\Gamma, W^{\dagger}(R_{\infty})^{\flat}))=q\Omega^{\dagger\bullet}_{A(R)^{\square}/A_{\inf}}\,.
\end{equation*}
We will show that the $q$-dagger de Rham complex computes the overconvergent prismatic cohomology after applying $\otimes^{\dagger}_{A_{\inf}}B_{\mathrm{cris}}^{+}$.

Let $A=A_{\inf}$ and let $R/\mathcal{O}$ ($\mathcal{O}=A/d$) be weakly complete, smooth. So $d=1+[\epsilon^{1/p}]+\cdots+[\epsilon^{1/p}]^{p-1}$ which is a generator of $\ker(A\rightarrow\mathcal{O})$. Then 
\begin{equation*}
\varphi(d):=[p]_{q}=\frac{q^{p}-1}{q-1}
\end{equation*} 
for $q=[\epsilon]$ which is the $q$-analogue of $p$. Let $\varphi^{\ast}R=R^{(1)}/(A/[p]_{q})$ be the base change of $R$ along $A\xrightarrow{\varphi}A$. Consider $P_{0}$, a  free $A$-algebra, weakly $(p,d)$-completed with an action of $\varphi$ and a surjection $P_{0}\twoheadrightarrow R$. Let $P_{\bigcdot}:=P_{0}\rightrightarrows P_{1}\mathrel{\substack{\textstyle\rightarrow\\[-0.4ex]
                      \textstyle\rightarrow \\[-0.4ex]
                      \textstyle\rightarrow}}\cdots
$ be the $(p,d)$-weakly completed \v{C}ech-nerve of $A\rightarrow P_{0}$, and let $J_{\bigcdot}=\ker(P_{\bigcdot}\rightarrow R)$ be the kernel of the augmentation and $P_{\bigcdot,\delta}$ the associated $\delta$-$P_{\bigcdot}$-algebra. Then $C_{i}:=P_{i,\delta}^{\dagger}\{\frac{J_{i}}{d}\}$ is a dagger prism ($\dagger$ = $(p,d)$-weakly completed). We know that $C_{\bigcdot}$, the $(p,d)$-weakly completed \v{C}ech-nerve of $A\rightarrow C_{0}$ computes $\Prism_{R/A_{\inf}}^{\dagger}$ equipped with its $\varphi$-action. Then
\begin{align*}
\varphi^{\ast}(C_{\bigcdot})
& =\varphi_{A}^{\ast}(C_{\bigcdot})=\varphi_{P_{\bigcdot,\delta}}^{\ast}(C_{\bigcdot}) \\
& =\varphi_{P_{\bigcdot,\delta}}^{\ast}\left(P_{\bigcdot,\delta}^{\dagger}\left\{\frac{J_{\bigcdot}}{d}\right\}\right)=P_{\bigcdot,\delta}^{\dagger}\left\{\frac{\varphi(J_{\bigcdot})}{[p]_{q}}\right\}
\end{align*} 
where we have used that $A\rightarrow P_{\bigcdot,\delta}$ is a cosimplicial homotopy equivalence. Then we have a ($(p,[q]_{p})$-weakly completed) $[p]_{q}$-PD version of Berthelot's Poincar\'{e} lemma \cite[Lemma V 2.1.2]{Ber74}:

\begin{prop} $\varphi^{\ast}(C_{n})=P_{n,\delta}^{\dagger}\left\{\frac{\varphi(J_{n})}{[p]_{q}}\right\}$ is a $(p,[q]_{p})$-weakly completed $[p]_{q}$-PD-polynomial algebra over $P_{0,\delta}^{\dagger}\left\{\frac{\varphi(J_{0})}{[p]_{q}}\right\}$, so for $N:=P_{0,\delta}^{\dagger}\left\{\frac{\varphi(J_{0})}{[p]_{q}}\right\}$ we have
\begin{equation*}
\varphi^{\ast}(C_{\bigcdot})\cong N^{\dagger}\langle T_{1},\ldots, T_{s}\rangle\left\langle\frac{T_{1}^{p}}{[p]_{q}},\ldots, \frac{T_{s}^{p}}{[p]_{q}}\right\rangle\,.
\end{equation*}
\end{prop}

\begin{lemma}
Let $S$ be a $(p,[p]_{q})$-weakly complete $\mathbb{Z}_{q}\llbracket q-1\rrbracket$-algebra. Then the $q$-dagger de Rham complex 
\begin{equation*}
q\Omega^{\dagger\bullet}_{S\langle T_{1},\ldots, T_{s}\rangle\left\langle\frac{T_{1}^{p}}{[p]_{q}},\ldots, \frac{T_{s}^{p}}{[p]_{q}}\right\rangle}
\end{equation*}
is acyclic, where the derivative is given by the $q$-derivation $\nabla_{q,n}:T^{n}\mapsto [n]_{q}T^{n-1}$.
\end{lemma}
\begin{proof}
Before taking the $(p,[p]_{q})$ weak completion, the proof is analogous to the proof that the usual de Rham complex of a PD-polynomial algebra is acyclic. Taking $(p,[p]_{q})$ weak completions preserves acyclicity. 
\end{proof}

\begin{cor}\label{corollary}
$q\Omega^{\dagger\bullet}_{\varphi^{\ast}(C_{n})/A}$ is quasi-isomorphic to $q\Omega^{\dagger\bullet}_{\varphi^{\ast}(C_{0})/A}$.
\end{cor}

We observe that $q\Omega^{\dagger\bullet}_{\varphi^{\ast}(C_{0})/A}$ is the $q$-dagger de Rham complex over the $(p,[p]_{q})$-weakly completed $q$-PD-envelope $D^{\dagger}_{J_{0},q}(P_{0,\delta})$, which is the Koszul complex $K_{D^{\dagger}_{J_{0},q}(P_{0,\delta})}(\nabla_{q,1},\ldots,\nabla_{q,s})$. But since we have a smooth lifting $A(R)^{\square}$ of $R$ over $A$, we have also considered the $q$-de Rham complex $q\Omega^{\dagger\bullet}_{A(R)^{\square}/A_{\inf}}=K_{A(R)^{\square}}(\nabla_{q,1},\ldots,\nabla_{q,s})$. We explain the relationship between these complexes.

Let again $C_{\bigcdot}$ be the \v{C}ech-Alexander complex that computes $\Prism_{R/A}^{\dagger}$ and let $\varphi^{\ast}C_{\bigcdot}=P_{\bigcdot,\delta}^{\dagger}\left\{\frac{\varphi(J_{\bigcdot})}{[p]_{q}}\right\}=\Prism_{R^{(1)}/A}^{\dagger}$. By the argument in \cite[Theorem 2.12]{BdJ11}, using Corollary \ref{corollary}, the total complex of the simplicial complex $\varphi^{\ast}C_{\bigcdot}$ is quasi-isomorphic to the $q$-dagger de Rham complex
\begin{equation*}
q\Omega^{\dagger\bullet}_{\varphi^{\ast}(C_{0})/A}=K_{D^{\dagger}_{J_{0},q}(P_{0})}(\nabla_{q,1},\ldots,\nabla_{q,s})\,.
\end{equation*}
Now $\varphi^{\ast}(C_{\bigcdot})=A_{\varphi}\otimes_{A}C_{\bigcdot}\cong C_{\bigcdot}$ with 
\begin{align*}
C_{i}
& = A^{\dagger}[X_{1},\ldots, X_{d},T_{1},\ldots, T_{s}]/\langle g_{1}-T_{1}d,\ldots, g_{s}-T_{s}d\rangle \\
& \cong\widetilde{R}^{\dagger}\langle T_{1},\ldots, T_{s}\rangle
\end{align*}
where $\widetilde{R}$ is a dagger lift of $R$ over $A$, previously denoted by $A(R)^{\square}$. The isomorphism comes from the uniqueness of weak formalisations \cite[Theorem 3.3]{MW68}. By \cite[Corollary 12.4]{BMS18} the $q$-dagger de Rham complex $q\Omega^{\dagger\bullet}_{\widetilde{R}^{\dagger}\langle T_{1},\ldots, T_{s}\rangle/A}$ becomes the usual de Rham complex after $\otimes^{\dagger}_{A_{\inf}}B_{\mathrm{cris}}^{+}$. Hence we get
\begin{align*}
q\Omega^{\dagger\bullet}_{\widetilde{R}^{\dagger}\langle T_{1},\ldots, T_{s}\rangle/A}\otimes^{\dagger}_{A_{\inf}}B_{\mathrm{cris}}^{+}
& = \Omega^{\dagger\bullet}_{\widetilde{R}^{\dagger}\langle T_{1},\ldots, T_{s}\rangle/A}\otimes^{\dagger}_{A_{\inf}}B_{\mathrm{cris}}^{+} \\
& \cong \Omega^{\dagger\bullet}_{\widetilde{R}/A}\otimes^{\dagger}_{A_{\inf}}B_{\mathrm{cris}}^{+}\\
& \cong q\Omega^{\dagger\bullet}_{\widetilde{R}/A}\otimes^{\dagger}_{A_{\inf}}B_{\mathrm{cris}}^{+}
\end{align*}
where we have used again \cite[Corollary 12.4]{BMS18}. The symbol $\otimes^{\dagger}_{A_{\inf}}$ denotes the weakly completed tensor product with respect to the $p$-adic topology. Applying again \cite[Theorem 2.12]{BdJ11} in this situation we obtain 
\begin{equation*}
\mathrm{Tot}(C_{\bigcdot})\otimes^{\dagger}_{A_{\inf}}B_{\mathrm{cris}}^{+}\cong  q\Omega^{\dagger\bullet}_{\widetilde{R}/A}\otimes^{\dagger}_{A_{\inf}}B_{\mathrm{cris}}^{+}\,.
\end{equation*}

The final argument for the proof of Theorem \ref{main theorem} is then very similar to the proof of the comparison with Monsky-Washnitzer cohomology in the case $A=W(k)$. Namely, let $M^{r,s}=q\Omega^{\dagger r}_{\varphi^{\ast}(C_{s})/A}$. For $r=0$ we get the cosimplicial complex that computes $\varphi^{\ast}\Prism^{\dagger}_{R/A}$. For fixed $s$, the $q$-de Rham complex $q\Omega^{\dagger\bullet}_{\varphi^{\ast}(C_{s})/A}$ is quasi-isomorphic to $q\Omega^{\dagger\bullet}_{A(R)^{\square}/A}$ after $\otimes^{\dagger}_{A_{\inf}}B_{\mathrm{cris}}^{+}$. Moreover, for $r>0$ the cosimplicial complex $q\Omega^{\dagger r}_{\varphi^{\ast}(C_{\bigcdot})/A}$ is homotopic to zero (analogue of \cite[Lemma 2.15, Lemma 2.17]{BdJ11}) - compare with the proof of Lemma \ref{homotopic to zero}. The vertical totalisation of the simplicial bicomplex computes $A^{\dagger}\Omega_{R/\mathcal{O}}\otimes^{\dagger}_{A_{\inf}}B_{\mathrm{cris}}^{+}$, and the horizontal totalisation computes $\varphi^{\ast}\Prism_{R/A}^{\dagger}\otimes^{\dagger}_{A_{\inf}}B_{\mathrm{cris}}^{+}$.
\end{proof}

Since the $q$-derivation specialises to the usual derivation modulo $q-1$, the $q$-dagger de Rham complex satisfies 
\begin{equation*}
q\Omega^{\dagger\bullet}_{\varphi^{\ast}(C_{0})/A}\otimes^{L}A/[p]_{q}A\cong\Omega^{\dagger\bullet}_{R/\mathcal{O}}\,.
\end{equation*}
Hence we have
\begin{thm}\label{q-de Rham}
\begin{equation*}
\varphi^{\ast}\Prism_{R/A}^{\dagger}\otimes^{L}A/[p]_{q}A\cong\Omega^{\dagger\bullet}_{R/\mathcal{O}}\,.
\end{equation*}
This is the dagger de Rham comparison of overconvergent prismatic cohomology. It is the analogue of \cite[Theorem 1.8 (3)]{BS22}.
\end{thm}

\section{Comparison with an overconvergent de Rham-Witt complex}

In the final section we compare the complex $A^{\dagger}\Omega$ with an overconvergent de Rham-Witt complex, defined for a smooth weak formal $\mathcal{O}$-scheme $(\mathcal{O}=\mathcal{O}_{\mathbb{C}_{p}})$. 

As before, let $\mathrm{Spf}\,R$ be a small affine weak formal scheme over $\mathrm{Spf}\,\mathcal{O}$, $A_{\mathrm{inf}}(\mathcal{O})=W(\mathcal{O}^{\flat})$, equipped with the $p$-$d$-adic topology. Consider a dagger lifting $A(R)^{\square}$ of $R$ over $A_{\inf}(\mathcal{O})$ and let (for $m=\dim\mathrm{Spf}\,R)$)
\begin{equation*}
A(R)^{\square}\rightarrow(A(R)^{\square})^{m}\rightarrow(A(R)^{\square})^{m \choose 2}\rightarrow\ldots
\end{equation*}
be the Koszul complex, or $q$-dagger de Rham complex $q\Omega^{\dagger\bullet}_{A(R)^{\square}/A_{\inf}}$, considered in Section 5. It is a complex of $A_{\inf}$-modules, where each entry is weakly complete with respect to the $(p,d)$-adic topology. 

\begin{lemma}
Under the map $\theta_{\infty}:A_{\inf}(\mathcal{O})\rightarrow W(\mathcal{O})=\varprojlim W_{r}(\mathcal{O})$, the $(p,d)$-adic topology maps to the $p$-$V$-adic topology. More precisely, for $\xi$ generating $\ker(A_{\inf}(\mathcal{O})\rightarrow\mathcal{O})$ such that for $\theta_{r}:W(\mathcal{O}^{\flat})\rightarrow W_{r}(\mathcal{O})$, $r>1$ we have $\theta_{r}(\xi)=V(1)$ and $\theta_{r}$ maps $\xi\varphi^{-1}(\xi)\cdots\varphi^{-s}(\xi)$ to $V^{s+1}(1)$ for $s<r-1$.
\end{lemma} 
\begin{proof}
See \cite[Lemma 3.4, Lemma 3.12]{BMS18}. By \cite[Example 3.16]{BMS18} we can take $\xi=d$. We have $\varphi^{-i}(d)\in (p,d)$ for $i\in\mathbb{Z}$, because $\varphi(d)=[p]_{q}\equiv p\mod q-1\equiv p\mod d$ since $d|q-1$ (where $q=[\epsilon]$). Hence $\lambda_{1}d+p=\varphi(d)\in(p,d)$, so $\lambda_{1}\varphi^{-1}(d)+p=d\in(p,d)$ and thus $\varphi^{-1}(d)\in(p,d)$. By induction $\varphi^{-i}(d)\in(p,d)$ for all $i$. 
\end{proof}

Under the base change $\theta_{\infty}:A_{\inf}(\mathcal{O})\rightarrow W(\mathcal{O})$ with kernel $(\mu)=([\epsilon]-1)$, the Koszul complex ($q$-de Rham complex) 
\begin{equation*}
A(R)^{\square}\rightarrow(A(R)^{\square})^{m}\rightarrow(A(R)^{\square})^{m \choose 2}\rightarrow\ldots
\end{equation*}
becomes the usual de Rham complex over $W(\mathcal{O})$, so
\begin{equation*}
q\Omega^{\dagger\bullet}_{A(R)^{\square}/A_{\inf}}\otimes_{A_{\inf}}W(\mathcal{O})\cong\Omega^{\dagger\bullet}_{A(R)^{\square}\otimes W(\mathcal{O})/W(\mathcal{O})}
\end{equation*}
which is a complex of $p$-$V$-weakly complete $W(\mathcal{O})$-modules. Let $\tilde{R}:=A(R)^{\square}\otimes W(\mathcal{O})$, which is a lifting of $R$ over $W(\mathcal{O})$. 

Then $\Omega^{\dagger\bullet}_{\tilde{R}/W(\mathcal{O})}:=\Omega^{\dagger\bullet}_{A(R)^{\square}\otimes W(\mathcal{O})/W(\mathcal{O})}$ is a dagger de Rham complex for the lifting $\tilde{R}$ of $R$ over $W(\mathcal{O})$. In order to globalise the comparison with $A^{\dagger}\Omega$ we need to define an overconvergent de Rham-Witt complex. 

Let 
\begin{equation*}
W\Omega^{\bullet}_{\widehat{\mathcal{O}}[T_{1},\ldots, T_{m}]/\mathcal{O}}:=\varprojlim_{s,n}W_{s}\Omega^{\bullet}_{\mathcal{O}/p^{n}[T_{1},\ldots, T_{m}]/\mathcal{O}}
\end{equation*}
be the $p$-adically completed de Rham-Witt complex of \cite{LZ04}.

Let $z\in\mathcal{O}^{\dagger}\langle T_{1},\ldots, T_{m}\rangle$, $z=\sum_{\kappa}a_{\kappa}\underline{T}^{\kappa}$, which converges weakly with respect to the $p$-adic topology, and $\kappa$ runs through multi-indices in $\mathbb{N}_{0}^{m}$. Define for $\epsilon>0$ $\gamma_{\epsilon}(z)=\inf_{\kappa}\{v_{p}(a_{\kappa}-\epsilon|\kappa|\}$. Then $\gamma_{\epsilon}(z)>-\infty$ for some $\epsilon>0$.

Define $\underline{Y}=(Y_{0},Y_{1},\ldots)\in W(\widehat{\mathcal{O}}\langle T_{1},\ldots, T_{m}\rangle)$ to be in $W^{\dagger}(\mathcal{O}^{\dagger}\langle T_{1},\ldots, T_{m}\rangle)$ if there exists $\epsilon>0$ such that $\gamma_{\epsilon}(Y_{i})>-\infty$ for all $i$, and moreover, if $Y_{i}=\sum_{\kappa(i)}a_{\kappa(i)}\underline{T}^{\kappa(i)}$ we have
\begin{equation*}
\inf_{i,\kappa(i)}\{i+(v_{p}(a_{\kappa(i)}-\epsilon(\kappa(i))p^{-i})\}>-\infty\,,
\end{equation*}
and likewise for $W^{\dagger}\Omega^{\bullet}_{\mathcal{O}^{\dagger}\langle T_{1},\ldots, T_{m}\rangle/\mathcal{O}}$, using the unique description of an element in $W\Omega^{\bullet}_{\widehat{\mathcal{O}}\langle T_{1},\ldots, T_{m}\rangle/\mathcal{O}}$ as a $p$-$V$-adically convergent sum of basic Witt differentials in \cite{LZ04}.

We have the following
\begin{lemma}
Under the base change $\mathcal{O}\rightarrow k$, 
\begin{equation*}
W^{\dagger}\Omega^{\bullet}_{\mathcal{O}^{\dagger}\langle T_{1},\ldots, T_{m}\rangle/\mathcal{O}}\otimes^{L}_{W(\mathcal{O})}W(k)\simeq W^{\dagger}\Omega^{\bullet}_{k[T_{1}\ldots, T_{m}]/k}
\end{equation*}
which is the overconvergent de Rham-Witt complex of the closed fibre defined in \cite{DLZ11}.
\end{lemma}
\begin{proof}
This is clear because the condition 
\begin{equation*}
\inf_{i,\kappa(i)}\{i+(v_{p}(a_{\kappa(i)}-\epsilon(\kappa(i))p^{-i})\}>-\infty
\end{equation*}
becomes 
\begin{equation*}
\inf_{i}\{i-p^{-i}\epsilon\deg\overline{Y}_{i})\}=\gamma_{\epsilon}(\overline{Y}_{0},\overline{Y}_{1},\ldots)>-\infty
\end{equation*}
which is the growth condition on $W(k[T_{1},\ldots, T_{n}])$, and likewise for $W\Omega$.
\end{proof}

Then we have
\begin{lemma}\label{lemma poly}
There is a quasi-isomorphism of $p$-$V$-adically weakly complete complexes
\begin{equation*}
\Omega^{\bullet}_{W(\mathcal{O})^{\dagger}\langle T_{1},\ldots, T_{m}\rangle/\mathcal{O}}\simeq W^{\dagger}\Omega^{\bullet}_{\mathcal{O}^{\dagger}\langle T_{1},\ldots, T_{m}\rangle/\mathcal{O}}\,.
\end{equation*}
\end{lemma}
\begin{proof}
Decompose $W^{\dagger}\Omega^{\bullet}_{\mathcal{O}^{\dagger}\langle T_{1},\ldots, T_{m}\rangle/\mathcal{O}}$ into a direct sum of 
\begin{equation*}
W^{\dagger}\Omega^{\bullet,(\mathrm{int})}_{\mathcal{O}^{\dagger}\langle T_{1},\ldots, T_{m}\rangle/\mathcal{O}}:=\Omega^{\bullet}_{W(\mathcal{O})^{\dagger}\langle T_{1},\ldots, T_{m}\rangle/\mathcal{O}}
\end{equation*}
and a fractional part $W^{\dagger}\Omega^{\bullet,(\mathrm{frac})}_{\mathcal{O}^{\dagger}\langle T_{1},\ldots, T_{m}\rangle/\mathcal{O}}$. The direct sum decomposition holds at finite level \cite{LZ04} for $W_{s}\Omega^{\bullet}_{\mathcal{O}/p^{n}[T_{1},\ldots, T_{m}]/\mathcal{O}/p^{n}}$. Then take $p$-$V$-adic completion and identify the overconvergent elements in the integral and fractional part. Details are left to the reader. It is clear that $W^{\dagger}\Omega^{\bullet,(\mathrm{frac})}_{\mathcal{O}^{\dagger}\langle T_{1},\ldots, T_{m}\rangle/\mathcal{O}}$ is acyclic. 
\end{proof}

Now let $\tilde{R}$ be \'{e}tale over $W(\mathcal{O})^{\dagger}\langle T_{1},\ldots, T_{m}\rangle$, lifting $R/\mathcal{O}^{\dagger}\langle T_{1},\ldots, T_{m}\rangle$. We claim that the above Lemma \ref{lemma poly} extends to $\tilde{R}$, i.e. 
\begin{lemma}
There is a quasi-isomorphism
\begin{equation*}
\Omega^{\dagger\bullet}_{\tilde{R}/W(\mathcal{O})}\simeq W^{\dagger}\Omega^{\bullet}_{R/\mathcal{O}}\,. 
\end{equation*}
\end{lemma} 
\begin{proof}
We have $\mathcal{O}^{\dagger}\langle T_{1},\ldots, T_{m}\rangle=\varinjlim_{\epsilon}\widehat{\mathcal{O}}\langle p^{\epsilon}T_{1},\ldots, p^{\epsilon}T_{m}\rangle$ where $\widehat{\mathcal{O}}\langle p^{\epsilon}T_{1},\ldots, p^{\epsilon}T_{m}\rangle$ denotes the $p$-adically convergent power series in $T_{1},\ldots, T_{m}$ that converge on a ball of radius $p^{\epsilon}$. Then $R=\varinjlim_{\epsilon}R_{\epsilon}$, with $R_{\epsilon}$ \'{e}tale over $\widehat{\mathcal{O}}\langle p^{\epsilon}T_{1},\ldots, p^{\epsilon}T_{m}\rangle$, and $\tilde{R}=\varinjlim_{\epsilon}\tilde{R}_{\epsilon}$. By \'{e}tale base change 
\begin{equation*}
W_{s}\Omega^{\bullet}_{\mathcal{O}/p^{n}[p^{\epsilon}T_{1},\ldots, p^{\epsilon}T_{m}]}\otimes\tilde{R}_{\epsilon,s}/p^{n}=W_{s}\Omega^{\bullet}_{R_{\epsilon}/p^{n}}
\end{equation*}
where $\tilde{R}_{s,\epsilon}:=\tilde{R}_{\epsilon}\otimes_{W(\mathcal{O})}W_{s}(\mathcal{O})$, and we have again a direct sum decomposition into an integral part (isomorphic to $\Omega^{\bullet}_{\tilde{R}_{\epsilon,s}/W_{s}(\mathcal{O})}$) and an acyclic fractional part. By taking limits over $s,n$ we get the analogous decomposition for $W\Omega^{\bullet}_{R_{\epsilon}/\mathcal{O}}$ and likewise for $W^{\dagger}\Omega^{\bullet}_{R_{\epsilon}/\mathcal{O}}$, identifying $W^{\dagger}\Omega^{\bullet,\mathrm{int}}_{R_{\epsilon}/\mathcal{O}}$ with $\Omega^{\dagger\bullet}_{\tilde{R}_{\epsilon}/W(\mathcal{O})}$, and an acyclic fractional part. Taking direct limits over $\epsilon$ identifies $W^{\dagger}\Omega^{\bullet}_{R/\mathcal{O}}$ as a subcomplex in $W\Omega^{\bullet}_{\widehat{R}/\mathcal{O}}$ with integral part isomorphic to $\Omega^{\dagger\bullet}_{\tilde{R}/W(\mathcal{O})}$ and acyclic fractional part. 
\end{proof}

For a weak formal smooth $\mathcal{O}$-scheme $\mathcal{X}$ we can define $W^{\dagger}\Omega_{\mathcal{X}/\mathcal{O}}^{\bullet}$ by gluing local data arising in a covering of $\mathcal{X}$ by affine weak formal schemes $\mathrm{Spf}^{\dagger}S$ such that $S$ is \'{e}tale over a weakly completed polynomial algebra and using that $W\Omega^{\bullet}_{\mathcal{X}/\mathcal{O}}$ is a complex of sheaves. Then we have the following final comparison result
\begin{thm}\label{overconvergent comparison}
Let $\mathcal{X}$ be a weak formal smooth $\mathcal{O}$-scheme. Then
\begin{equation*}
A^{\dagger}\Omega_{\mathcal{X}/\mathcal{O}}\otimes^{L}_{A_{\inf}}W(\mathcal{O})\simeq W^{\dagger}\Omega^{\bullet}_{\mathcal{X}/\mathcal{O}}\,.
\end{equation*}
\end{thm} 


\begin{thebibliography}{9}
\bibitem[Ber74]{Ber74}
P. Berthelot, \emph{Cohomologie cristalline des sch\'{e}mas de caract\'{e}ristique $p>0$}, Lecture Notes in Mathematics, Vol. 407. Springer-Verlag, Berlin-New York, 1974. 604 pp.

\bibitem[Bha18a]{Bha18a}
B. Bhatt, \emph{Prismatic cohomology (Eilenberg Lectures at Columbia University)}, 2018.

\bibitem[Bha18b]{Bha18b}
B. Bhatt, \emph{Specializing varieties and their cohomology from characteristic $0$ to characteristic $p$}, Algebraic geometry: Salt Lake City 2015, 43--88,
Proc. Sympos. Pure Math., 97.2, Amer. Math. Soc., Providence, RI, 2018.

\bibitem[BdJ11]{BdJ11}
B. Bhatt, A. J. de Jong, \emph{Crystalline cohomology and de Rham cohomology}, preprint 2011, arXiv:1110.5001v1 [math.AG] 

\bibitem[BGR84]{BGR84}
S. Bosch, U. G\"{u}ntzer, R. Remmert, \emph{Non-Archimedean analysis. A systematic approach to rigid analytic geomtry}, Grundlehren der Mathematischen Wissenschaften, 261. Berlin etc.: Springer Verlag. XII, 436 p. DM 168.00 (1984).

\bibitem[BM21]{BM21}
B. Bhatt, A. Mathew, \emph{The arc-topology}, Duke Math. J. 170 (2021), no. 9, 1899--1988.

\bibitem[BMS18]{BMS18}
B. Bhatt, M. Morrow, P. Scholze, \emph{Integral $p$-adic Hodge theory},  Publ. Math. Inst. Hautes Études Sci. 128 (2018), 219–397.

\bibitem[BMS19]{BMS19}
B. Bhatt, M. Morrow, P. Scholze, \emph{Topological Hochschild homology and integral $p$-adic Hodge theory}, Publ. Math. Inst. Hautes \'{E}tudes Sci. 128 (2018), 219--397.

\bibitem[BS22]{BS22}
B. Bhatt, P. Scholze, \emph{Prisms and prismatic cohomology},  Ann. of Math. (2) 196 (2022), no. 3, 1135–1275.

\bibitem[BS23]{BS23}
B. Bhatt, P. Scholze, \emph{Prismatic $F$-crystals and crystalline Galois representations}, Camb. J. Math. 11 (2023), no. 2, 507--562.

\bibitem[\v{C}K19]{CK19}
K. \v{C}esnavi\v{c}ius, T. Koshikawa, \emph{The $A_{\mathrm{inf}}$-cohomology in the semistable case}, Compos. Math. 155 (2019), no. 11, 2039--2128.

\bibitem[CC98]{CC98}
F. Cherbonnier, P. Colmez, \emph{Repr\'{e}sentations  $p$-adiques surconvergentes}, Invent. Math. 133 (1998), no. 3, 581--611.

\bibitem[CN20]{CN20}
P. Colmez, W. Nizio\l, \emph{On $p$-adic comparison theorems for rigid analytic varieties, I}, M\"{u}nster J. Math. 13 (2020), no. 2, 445--507.

\bibitem[DK15]{DK15}
C. Davis, K. Kedlaya, \emph{Almost purity and overconvergent Witt vectors},  J. Algebra 422 (2015), 373--412.

\bibitem[DLZ11]{DLZ11} 
C. Davis, A. Langer, T. Zink, \emph{Overconvergent de Rham-Witt Cohomology}, Annals. Sc. Ec. Norm. Sup. {\bf{44}}, no. 2 (2011), 197-262.

\bibitem[DLZ12]{DLZ12} 
C. Davis, A. Langer, T. Zink, \emph{Overconvergent Witt vectors}, J. Reine Angew. Math. 668 (2012), 1--34.

\bibitem[FX13]{FX13}
L. Fourquaux, B. Xie, \emph{Triangulable $\mathcal{O}_{F}$-analytic $(\varphi_{q},\Gamma)$-modules of rank $2$}, Algebra Number Theory 7 (2013), no. 10, 2545--2592.

\bibitem[KR09]{KR09} 
M. Kisin, W. Ren, \emph{Galois representations and Lubin-Tate groups}, Doc. Math. 14 (2009), 441--461.

\bibitem[Kos22]{Kos22}
T. Koshikawa, \emph{Logarithmic prismatic cohomology I}, preprint 2022, arXiv:2007.14037v3 [math.AG].

\bibitem[KY23]{KY23}
T. Koshikawa, Z. Yao, \emph{Logarithmic prismatic cohomology II}, preprint 2023, arXiv:2306.00364v1 [math.AG]

\bibitem[LZ04]{LZ04}
A. Langer and T. Zink, \emph{De Rham-Witt cohomology for a proper and smooth morphism}, J. Inst. Math. Jussieu 3, (2004), 231--314.

\bibitem[Law18]{Law18}
N. Lawless, \emph{A comparison of overconvergent de Rham-Witt cohomology and rigid cohomology on smooth schemes}, preprint 2018, arXiv:1810.10059v1 [math.NT] 

\bibitem[Mar23]{Mar23}
S. Marks, \emph{Prismatic $F$-crystals and Lubin-Tate $(\varphi_{q},\Gamma)$-modules}, preprint 2023, arXiv:2303.07620v2 [math.NT] 

\bibitem[Mer72]{Mer72}
D. Meredith, \emph{Weak formal schemes}, Nagoya Math. J. 45 (1972), 1-38.

\bibitem[MW68]{MW68}
P. Monsky, G. Washnitzer, \emph{Formal cohomology. I.}, Ann. of Math. (2) 88 (1968), 181--217.

\bibitem[Mor20]{Mor16}
M. Morrow, \emph{Notes on the $A_{\inf}$-cohomology of Integral $p$-adic Hodge theory}, Simons Symp., Springer, Cham, 2020, 1–69.

\bibitem[Sch12]{Sch12}
P. Scholze, \emph{Perfectoid spaces}, Publ. Math. Inst. Hautes \'{E}tudes Sci. 116 (2012), 245--313.

\bibitem[Sch13]{Sch13}
P. Scholze, \emph{$p$-adic Hodge theory for rigid-analytic varieties}, Forum Math. Pi 1 (2013), e1, 77 pp.

\bibitem[Stacks]{Stacks}
The Stacks Project Authors, \emph{Stacks Project}, https://stacks.math.columbia.edu, 2022. 

\bibitem[Wu21]{Wu21}
Z. Wu, \emph{Galois representations, $(\varphi,\Gamma)$-modules and prismatic $F$-crystals}, Doc. Math, 26 (2021), 1771--1798.


\end{thebibliography}
\end{document}